\documentclass[11pt]{article}
\usepackage[T1]{fontenc}
\usepackage[utf8]{inputenc}
\usepackage{amsmath,amssymb,amsthm,mathtools,mathrsfs}
\usepackage{enumitem}
\usepackage{needspace}
\usepackage{cite}
\usepackage[hidelinks]{hyperref}
\usepackage{geometry}
\hypersetup{
 pdftitle={Cyclic source pairings for Penrose--Sparling non-Hausdorff twistor spaces},
 pdfauthor={Ioannis P. Zois}}
\mathtoolsset{showonlyrefs}
\newcommand{\C}{\mathbb C}
\newcommand{\Z}{\mathbb Z}
\newcommand{\R}{\mathbb R}
\newcommand{\CP}{\mathbb{CP}}

\newcommand{\calG}{\mathcal G}
\newcommand{\calO}{\mathcal O}

\newcommand{\ch}{\operatorname{ch}}
\newcommand{\Tr}{\operatorname{Tr}}
\newcommand{\id}{\operatorname{id}}
\newcommand{\Tot}{\operatorname{Tot}}
\newcommand{\End}{\operatorname{End}}
\newcommand{\Hom}{\operatorname{Hom}}
\newcommand{\diag}{\operatorname{diag}}
\newcommand{\Gr}{\operatorname{Gr}}
\newcommand{\rk}{\operatorname{rk}}
\newcommand{\Td}{\operatorname{Td}}
\newcommand{\im}{\operatorname{im}}

\newcommand{\rel}{\mathrm{rel}}
\newcommand{\alg}{\mathscr A}

\newcommand{\Src}{\mathfrak S}
\newcommand{\Res}{\mathfrak R}
\newcommand{\wt}{\mathrm{wt}}

\theoremstyle{plain}
\newtheorem{Theorem}{Theorem}[section]
\newtheorem{Proposition}[Theorem]{Proposition}
\newtheorem{Lemma}[Theorem]{Lemma}
\newtheorem{Corollary}[Theorem]{Corollary}
\theoremstyle{definition}
\newtheorem{Definition}[Theorem]{Definition}
\newtheorem{Remark}[Theorem]{Remark}
\newtheorem{Example}[Theorem]{Example}
\numberwithin{equation}{section}

\title{Cyclic source pairings for Penrose--Sparling non-Hausdorff twistor spaces}
\author{Ioannis P.~ZOIS, Exeter College, Turl Street, OX1 3DP, Oxford,\\
  email: i.zois@exeter.oxon.org}
\date{}

\begin{document}
\maketitle

\begin{abstract}
We determine the stable source information carried by the gluing-groupoid
algebra of the Penrose--Sparling non-Hausdorff twistor double. Relative
normal residues define supported analytic $K$-classes with specified
normal orientations. For a smooth double along a closed submanifold, we
compute the full geometric periodic cyclic homology as the parity sum of
the compactly supported cohomologies of the ambient manifold and the
doubled locus. In the physical model, $K_0\cong\Z^4$ and $HP_0\cong\C^4$,
with both odd groups zero. Four explicit cocycles give integral
$K$-coordinates and exhaust continuous periodic cyclic cohomology;
every periodic cyclic pairing on supported source classes is therefore a
linear combination of rank and degree. Finite moments recover
$K$-valued normal profiles modulo integral common line twists.
A first-jet criterion tests mixed-line triviality and identifies the
residue with the torsion cokernel of a generically invertible second
fundamental form for pullback quotients. The degree-one
Euler modification and the diagonal Coulomb embedding have identical
filtered $K$-profiles in every algebraic $SL_2$ representation, while their rank-two Ward
connections are respectively indecomposable and split. We give the Euler
connection and its transition between both root charts. An elementary
holonomy realization shows that the inherited ambient and divisor
transverse classes annihilate supported source classes. The relative
degree pairing uses an additional choice of opposite defect orientations.
\end{abstract}

\noindent\textbf{Keywords:} noncommutative geometry; twistor spaces; cyclic cohomology;
operator-algebra $K$-theory; Penrose--Ward transform; source invariants.

\noindent\textbf{2020 Mathematics Subject Classification:} 32L25; 32L05; 46L80; 46L87; 19D55; 53C07.

\section{Introduction}
\label{sec:intro}
What does noncommutative geometry (NCG) retain of the source data in a
non-Hausdorff twistor space? We answer this question for the
Penrose--Sparling double by computing its analytic $K$-theory and its full
geometric periodic cyclic theory, then comparing these stable invariants
with mixed-line splitting and Ward reconstruction. The answer is partial
in a precise sense: the algebra retains the two divisor evaluations and
has a complete four-coordinate cyclic pairing, but every periodic cyclic
observable of the supported source class reduces to rank and degree.
The Euler modification exhibits the extension data lost even by the full
filtered $K$-profile.

Our test case has physical geometry
$X_0=\Tot(\calO_{\CP^1}(1)^{\oplus2})$ and
$D_0\cong\CP^1\times\C$
\cite{PenroseSparling1979,HughstonMason1990,GarnerPaquette2025}.
Scalar continuous functions identify the two non-separated divisor
points, whereas the gluing groupoid keeps them distinct. For a closed
subset $D$ of a second-countable locally compact Hausdorff space $X$, its
algebra is
\begin{equation}\label{eq:main-algebra}
 A_{X,D}=\left\{
 \begin{pmatrix}f_+&g\\h&f_-\end{pmatrix}\in M_2(C_0(X)):
 g|_D=h|_D=0\right\}.
\end{equation}
Relative normal residues and supported induction give a defined route
from meromorphic gluing to analytic classes. Direct restriction to the
physical source line vanishes; supported induction uses its normal Thom
class. In the physical model, the four cocycles of
Theorem~\ref{thm:complete-pairing} give integral coordinates on
$K_0(A_0)\cong\Z^4$ and identify the sheet-anti-invariant source summand.
Theorem~\ref{thm:full-periodic-cyclic} computes the entire periodic cyclic
theory in the specified smooth convention and proves that these cocycles
form a basis. Corollary~\ref{cor:cyclic-source-exhaustion} consequently
exhausts the stable cyclic source observables. These results use the
general groupoid theory of \cite{Crainic1999,CrainicMoerdijk2000,
CrainicMoerdijk2001}; the numerical formulas use standard Thom and
Chern-character identities with the source input specified.

For locally split constant normal type, finite cyclic moments reconstruct
the $K$-valued profile modulo integral common line twists. The geometric
comparison is sharper than nonvanishing of a field with vanishing
residue. A first-jet criterion treats balanced modifications with an
arbitrary plus bundle. For pullback quotients, a generically invertible
second fundamental form has torsion cokernel representing the residue;
its length gives the degree $2d-2$ in rank two. The degree-one Euler
modification and the diagonal Coulomb embedding have identical full
filtered profiles in every algebraic $SL_2$ representation, yet their
rank-two Ward connections are locally inequivalent. The nonzero curvature
commutator detects indecomposability, and an explicit two-chart
transition covers the chosen-root space off the discriminant. Thus the
loss occurs in passage to stable profiles, before the cyclic pairing is
taken.

The original motivation was a discussion with Roger Penrose about NCG
and non-Hausdorff twistor geometry, together with the author's
Noncommutative Foliation Invariant (NCFI)
\cite{Zois2000,ZoisNCFI2026}. In even codimension its pairing has the form
\begin{equation}\label{eq:ncfi-pattern}
 Z(\mathcal F)=\langle\varphi_{\mathcal F},[e_{\mathcal F}]\rangle,
\end{equation}
with a transverse fundamental cyclic class and a specified transverse
geometric module. Its operator-algebraic background is developed in
\cite{Connes1980,ConnesSurvey1982,Connes1985,ConnesTF1986,Connes1994,
ConnesSkandalis1984,ConnesMoscovici1995}.
Section~\ref{sec:foliation} uses an elementary holonomy realization of
the double. With the inherited complex transverse orientations, both
the ambient and divisor fundamental classes annihilate supported source
classes. Corollary~\ref{cor:cyclic-sheet-symmetry} identifies the
annihilator of the supported subgroup as exactly the sheet-invariant
sector of periodic cyclic cohomology. The relative divisor pairing
requires opposite defect
orientations as additional data. The realization is noncompact and no
designated transverse geometric module class is constructed here;
neither the NCFI nor a relative NCFI is defined or computed.

Garner--Paquette supply the double, Coulomb bundle, cocharacter backgrounds
and mixed lifts \cite[Sections~2--3]{GarnerPaquette2025}.
Marcolli--Penrose encode Donaldson--Friedman twistor gluing in an
assembled convolution algebra and study its diagram deformation theory
\cite[arXiv~v1, Corollary~3.2 and Theorem~3.3, p.~29]{MarcolliPenrose2020}.
Our subject is the supported analytic source construction, the complete
periodic cyclic calculation, and their geometric limits. Related
frameworks include sourced relative cohomology
\cite{Bailey1985,BaileySinger1989,WoodhouseMason1988}, the Atiyah--Ward
ansatz \cite{AtiyahWard1977,WardWells1990}, lattice and filtered-functor
constructions \cite{Zhu2016,Ziegler2015}, fanning curves
\cite{AlvarezPaivaDuran2009}, and cocycle-twisted twistor theory
\cite{BrainMajid2008,MeierViana2026}. The ordered sheets, support cycle,
complex orientations and relative or filtered input are fixed throughout.

\section{Physical twistor space, compactification, and Ward data}
\label{sec:geometry}
\subsection{The two ambient spaces}
Write $B=\CP^1$ and $M=\calO_B(1)$.  The flat-space twistor space considered
here is
\begin{equation}\label{eq:physical-X}
 X_0=\Tot(M\oplus M)\xrightarrow{p}B.
\end{equation}
It is the complement of a projective line $I_\infty$ in
$\overline X=\CP^3$.  In homogeneous coordinates
$[\lambda_0:\lambda_1:\mu_0:\mu_1]$, one can take
\[
 I_\infty=\{\lambda_0=\lambda_1=0\},\qquad
 q=\lambda_0\mu_1-\lambda_1\mu_0.
\]
These coordinates give the straight-source quadric model. Write $M$
also for its pullback to $X_0$. The section $q\in H^0(X_0,M^2)$ defines
the physical divisor $D_0$, whose compact closure
$\overline D=\{q=0\}\subset\CP^3$ is a smooth quadric.
We follow \cite[Section~2]{GarnerPaquette2025} for the flat-space
convention; see \cite{Penrose1967,PenroseRindler1,PenroseRindler2,
WardWells1990,AtiyahDunajskiMason2017} for background.

\begin{Lemma}\label{lem:physical-divisor}
The divisor and its normal bundle are
\[
 D_0\cong B\times\C,\qquad N_{D_0/X_0}\cong p^*\calO_B(2).
\]
Moreover $I_\infty\subset\overline D$, and
$X_0\setminus D_0=\overline X\setminus\overline D$.
\end{Lemma}
\begin{proof}
The surjective bundle map
\[
 \calO_B(1)\oplus\calO_B(1)\longrightarrow\calO_B(2),
 \qquad (\mu_0,\mu_1)\longmapsto\lambda_0\mu_1-\lambda_1\mu_0
\]
has kernel of degree $1+1-2=0$, hence $\calO_B$, with total space
$D_0$. Its quotient gives the normal bundle. The remaining assertions
follow from the equation of $q$.
\end{proof}

Let $L\subset D_0$ be the zero section $B\times\{0\}$.  Its normal bundle
in $D_0$ is trivial.  In contrast,
\[
 \overline D\cong\CP^1\times\CP^1,\qquad
 H|_{\overline D}=h_1+h_2,\qquad
 \int_{\overline D}H^2=2,
\]
where $H=c_1(\calO_{\CP^3}(1))$.  Although $L$ is compact, $D_0$ is not.
In particular $H^4(D_0;\Z)=0$, whereas $H_c^4(D_0;\Z)\cong\Z$.
The distinction between ordinary and compactly supported cohomology will be
essential; it does not prohibit supported degree-four pairings.

\subsection{Ward data and mixed lifts}
For spacetime duality, write
$x=\left(\begin{smallmatrix}a_0&a_1\\b_0&b_1\end{smallmatrix}\right)$
and use
\begin{equation}\label{eq:spacetime-duality-convention}
 g=\det(dx)=da_0\,db_1-da_1\,db_0,\qquad
 \operatorname{vol}_g=-\tfrac14
 da_0\wedge da_1\wedge db_0\wedge db_1.
\end{equation}
Here $\alpha\,\beta=(\alpha\otimes\beta+\beta\otimes\alpha)/2$. We call a
complex curvature anti-self-dual (ASD) when $*_gF=-F$. In these coordinates
this is equivalent to
$F_{a_0b_0}=F_{a_1b_1}=0$ and
$F_{a_1b_0}+F_{a_0b_1}=0$.

The source-free Penrose--Ward correspondence relates suitable ASD
connections to holomorphic bundles trivial on twistor lines, with the
appropriate reality conditions
\cite{Ward1977,AtiyahWard1977,AtiyahHitchinSinger1978,WardWells1990,
MasonWoodhouse1996}. Related developments include the instanton construction
\cite{AtiyahDrinfeldHitchinManin1978}, the cohomological Penrose transform
\cite{EastwoodPenroseWells1981,BastonEastwood1989}, and nonlinear-graviton
and integrable-systems constructions \cite{Penrose1976,Ward1985}.

For the present sourced model, the charge-$n$ holomorphic gluing data are
\begin{equation}\label{eq:symmetric-coulomb}
 E_{n,+}=M^n,\qquad E_{n,-}=M^{-n},\qquad
 \alpha_n=q^n:E_{n,-}|_U\longrightarrow E_{n,+}|_U,
 \quad U=X_0\setminus D_0.
\end{equation}
The inverse gluing is $q^{-n}$. For $n=1$ this is the Coulomb bundle
of \cite[Section~3.2]{GarnerPaquette2025}, trivial on the admissible
lifted spheres; integer powers give the abelian family. Since $q$ is a
section of $M^2$, the gluing is a bundle morphism. Expressions such as
$q^{-\lambda}$ mean locally trivialized morphisms, not a global map
$U\to G$.

A common twist by $M^a$ changes a trivial lifted-line restriction to
$\calO_{\CP^1}(a)$ and can destroy Ward admissibility. Thus the
asymmetric pair $(\calO,M^{-2n})$, obtained by twisting
\eqref{eq:symmetric-coulomb} by $M^{-n}$, has different line triviality
for $n\ne0$. Section~\ref{sec:mixed-lines} tests mixed restrictions and
reconstructs the Euler member. A general real sourced correspondence
would also require specified singularities, reality conditions, and a
converse reconstruction theorem.

\section{The doubled-divisor algebra and its smooth core}
\label{sec:algebra}
Let $X$ be second countable, locally compact, and Hausdorff, and let $D$ be
closed.  The double is formed from $X_+\sqcup X_-$ by identifying corresponding
points of $U=X\setminus D$.  Its presenting groupoid has arrow space
\[
 \calG_{X,D}=X_+\sqcup X_-\sqcup U_{+-}\sqcup U_{-+},
\]
with the evident structure maps. The groupoid is \'etale with Hausdorff
arrow space. If $D$ contains a limit point of $U$, its orbit space is
non-Hausdorff and the groupoid is nonproper. For groupoids and Morita
equivalence see \cite{Renault1980,MuhlyRenaultWilliams1987,Paterson1999,
MoerdijkMrcun2003,Tu2004}.

\begin{Proposition}\label{prop:algebra}
The reduced convolution algebra is the algebra $A_{X,D}$ of
\eqref{eq:main-algebra}.  It fits into the exact sequence
\begin{equation}\label{eq:extension}
 0\longrightarrow M_2(C_0(U))\longrightarrow A_{X,D}
 \xrightarrow{(\rho_+,\rho_-)}C_0(D)\oplus C_0(D)
 \longrightarrow0.
\end{equation}
It is also the pullback of restriction
$M_2(C_0(X))\to M_2(C_0(D))$ and the diagonal inclusion
$C_0(D)^2\to M_2(C_0(D))$.
\end{Proposition}
\begin{proof}
On each two-point orbit in $U$, convolution and the regular norm are
matrix multiplication and matrix norm. At $D$, the singleton orbits
give the two diagonal scalar representations. Completing the diagonal
$C_c(X)$ entries and off-diagonal $C_c(U)$ entries, extended by zero,
therefore gives \eqref{eq:main-algebra}. Restriction has kernel
$M_2(C_0(U))$ and is onto because $C_0(X)\to C_0(D)$ is onto.
These descriptions also identify the pullback.
\end{proof}

We use the reduced algebra with conventions from
\cite{Davidson1996,Williams2007}. It is type I, with irreducible
representations of dimension at most two; related treatments of
non-Hausdorff spectra and Fell algebras appear in
\cite{RaeburnWilliams1998,AnHuefKumjianSims2011,DeeleyGoffengYashinski2022}.

\begin{Lemma}[Stable projectionlessness]\label{lem:projectionless}
If $X$ is connected and noncompact, then $A_{X,D}$ is stably
projectionless. In particular this holds for the physical model
$A_{X_0,D_0}$.
\end{Lemma}
\begin{proof}
A projection in
$A_{X,D}\otimes\mathcal K\subset C_0(X,M_2\otimes\mathcal K)$
is a norm-continuous field of compact-operator projections. Its finite rank
is locally constant, hence constant because $X$ is connected. Vanishing
at infinity forces the projection norm to be less than one at some point
of the noncompact space, so this constant rank is zero.
\end{proof}

Consequently nonzero physical $K_0$-classes must be represented by
differences of projections in a unitization. The supported construction
in Section~\ref{sec:induction} has this form.

Suppose now that $X$ is a smooth manifold and $D$ is a closed smooth
submanifold.  A convenient smooth algebra is
\begin{equation}\label{eq:smooth-core}
 \alg_{X,D}=
 \left\{\begin{pmatrix}f_+&g\\h&f_-\end{pmatrix}:
 f_\pm\in C_c^\infty(X),\quad g,h\in C_c^\infty(U)\right\}.
\end{equation}
For compact $X$, $C_c^\infty(X)=C^\infty(X)$.  The off-diagonal entries
are extended by zero, so each vanishes on a neighbourhood of $D$.  We equip
this algebra with the usual inductive-limit smooth topology over fixed compact
supports, with the off-diagonal supports kept inside $U$.

\begin{Lemma}\label{lem:spectral-invariance}
The algebra \eqref{eq:smooth-core} is dense in $A_{X,D}$ and is stable under
holomorphic functional calculus after unitization and matrix amplification.
Its topological $K$-theory agrees with that of $A_{X,D}$.
\end{Lemma}
\begin{proof}
Density follows by smoothing and by cutting off off-diagonal functions away
from $D$.  An invertible element of a matrix unitization is a smooth matrix
function on $X$, with a fixed scalar matrix at infinity in the noncompact
case.  Its pointwise inverse is smooth.  Outside the union of the compact
supports of its off-diagonal blocks, the matrix is block diagonal, and its
inverse is block diagonal as well.  Outside the compact supports of all
nonconstant entries, the inverse is constant.  Thus the inverse belongs to
the same matrix unitization.  Holomorphic functional calculus follows from
the resolvent formula, with all nonconstant supports in one fixed compact
set.  Matrix spectral invariance, or equivalently approximation of projections
and their homotopies followed by functional calculus, gives the asserted
$K$-theory isomorphism; see \cite{Blackadar1998,Connes1994,Khalkhali2009}.
\end{proof}

For general groupoid cyclic cohomology and excision see
\cite{Crainic1999,CrainicMoerdijk2000,Cuntz1997,CuntzQuillen1995,
CuntzQuillen1997,Loday1998}. Here the cocycles use explicit divisor
restrictions and ambient full matrix-trace integration.

\section{An explicit analytic \texorpdfstring{$K$}{K}-theory decomposition}
\label{sec:K-theory}
Write $K_c^\nu(X)=K_\nu(C_0(X))$. Ordinary matrix Morita invariance
identifies $K_\nu(M_2(C_0(X)))$ with $K_c^\nu(X)$. Denote the algebra
inclusion by $\iota:A_{X,D}\hookrightarrow M_2(C_0(X))$ and the closed
embedding by $i:D\hookrightarrow X$.

\begin{Theorem}\label{thm:K-splitting}
For $\nu=0,1$, the homomorphism
\begin{equation}\label{eq:Theta}
 \Theta_\nu=(\iota_*,\rho_{+*}):K_\nu(A_{X,D})
 \longrightarrow K_c^\nu(X)\oplus K_c^\nu(D)
\end{equation}
is an isomorphism.  If $\Theta_\nu(z)=(x,a)$, then
\begin{equation}\label{eq:rho-minus}
 \rho_{-*}(z)=i^*x-a.
\end{equation}
The isomorphism is natural for compatible proper maps of pairs that induce
the displayed pullback maps; it depends on the labeling of the two sheets.
\end{Theorem}
\begin{proof}
Apply the $K$-theory Mayer--Vietoris sequence to the pullback in
Proposition~\ref{prop:algebra}.  One leg, restriction from $X$ to $D$, is
surjective.  The relevant maps, under matrix Morita invariance, are
\[
 K_c^\nu(X)\oplus K_c^\nu(D)\oplus K_c^\nu(D)
 \xrightarrow{\delta_\nu}K_c^\nu(D),\qquad
 \delta_\nu(x,a,b)=i^*x-a-b.
\]
The diagonal map adds the two block classes. Each $\delta_\nu$ is
onto (take $x=b=0$), so both connecting maps vanish and exactness
identifies $K_\nu(A_{X,D})$ with $\ker\delta_\nu$. Projection to
$(x,a)$ has inverse $(x,a)\mapsto(x,a,i^*x-a)$, proving
\eqref{eq:Theta}--\eqref{eq:rho-minus}. Naturality follows from the
pullback sequence \cite[Theorem~4.5, p.~431]{Schochet1984}; see also
\cite{Blackadar1998}.
\end{proof}

\begin{Corollary}\label{cor:lift}
A pair $(a,b)\in K_c^\nu(D)^2$ is the divisor restriction of an analytic class
if and only if $a+b\in\im i^*$.  For a fixed total class $x$ satisfying
$i^*x=a+b$, its lift is unique.
\end{Corollary}
\begin{proof}
This is \eqref{eq:rho-minus} and injectivity of $\Theta_\nu$.
\end{proof}

This is an additive decomposition with restriction maps. For $s$ sheets
and all off-diagonal entries vanishing on $D$, the same proof gives
$K_c^\nu(X)\oplus K_c^\nu(D)^{s-1}$.

\subsection{The physical model and the restriction map}
Let $A_0=A_{X_0,D_0}$ and put $\epsilon=[M]-1\in K^0(B)$, so
$K^0(B)=\Z[\epsilon]/(\epsilon^2)$.  The complex Thom isomorphism identifies both
$K_c^0(X_0)$ and $K_c^0(D_0)$ with $K^0(B)$, and both odd groups vanish.
Under these identifications, restriction is
\begin{equation}\label{eq:physical-restriction}
 i^*:K^0(B)\longrightarrow K^0(B),\qquad
 \xi\longmapsto\xi(1-[M^{-2}])=2\epsilon\xi.
\end{equation}
To see this, split the exact sequence in Lemma~\ref{lem:physical-divisor}
as smooth complex bundles.  Restricting the Thom class of
$\calO\oplus M^2$ to the subbundle $\calO$ multiplies its Thom class by the
$K$-theoretic Euler class $\lambda_{-1}(M^{-2})$.  This is the usual
self-intersection formula \cite{AtiyahKTheory1967,Fulton1998}.
Theorem~\ref{thm:K-splitting} now gives
\begin{equation}\label{eq:physicalK}
 K_0(A_0)\cong\Z^4,\qquad K_1(A_0)=0.
\end{equation}
The matrix of \eqref{eq:physical-restriction} is
$\left(\begin{smallmatrix}0&0\\2&0\end{smallmatrix}\right)$.
The exact sequence for $U=X_0\setminus D_0$ therefore gives
\begin{equation}\label{eq:complementK}
 K_0(C_0(U))\cong\Z,\qquad
 K_1(C_0(U))\cong\Z\oplus\Z/2.
\end{equation}
Indeed, these groups are respectively the kernel and cokernel of
\eqref{eq:physical-restriction}, with kernel generated by $\epsilon$.
After Thom identification a divisor pair
$(a_0+a_1\epsilon,b_0+b_1\epsilon)$ lifts if and only if
\[
 a_0+b_0=0,\qquad a_1+b_1\equiv0\pmod2.
\]
The boundary obstruction is their sum in
$\Z[\epsilon]/(\epsilon^2,2\epsilon)\cong\Z\oplus\Z/2$.  Its torsion summand is invisible to complex-valued cyclic pairings,
since $\operatorname{Hom}(\Z/2,\C)=0$.

\subsection{Comparison with the compact twistor model}
Put $\overline A=A_{\overline X,\overline D}$.  Standard projective-space
$K$-theory gives
\[
 K^0(\CP^3)=\Z[h]/(h^4),\qquad h=[\calO(1)]-1,
\]
\[
 K^0(\overline D)=\Z[a,b]/(a^2,b^2),
 \quad a=[\calO(1,0)]-1,\quad b=[\calO(0,1)]-1,
\]
and odd groups vanish \cite{AtiyahKTheory1967,Blackadar1998}.
Since $[\calO(1)|_{\overline D}]=(1+a)(1+b)$, restriction sends
\begin{equation}\label{eq:compact-restriction}
 1\mapsto1,\qquad h\mapsto a+b+ab,\qquad
 h^2\mapsto2ab,\qquad h^3\mapsto0.
\end{equation}
Consequently,
\begin{equation}\label{eq:compactK}
 K_0(\overline A)\cong\Z^8,\qquad K_1(\overline A)=0.
\end{equation}
By \eqref{eq:compact-restriction}, a divisor pair lifts precisely when
its sum has equal $a$ and $b$ coefficients and its $ab$ coefficient has
the same parity as that common coefficient.

The complement $U=\CP^3\setminus\overline D$ is the same as
$X_0\setminus D_0$. The compact restriction matrix has Smith normal form
$\diag(1,1,2,0)$, with kernel generated by $h^3$ and cokernel
$\Z\oplus\Z/2$. Thus it reproduces \eqref{eq:complementK}, even though
\eqref{eq:compactK} and \eqref{eq:physicalK} differ. The ambient
compactification changes the analytic $K$-group and its divisor restriction
map while leaving the complement fixed.

\section{Support and common-twist obstructions}
\label{sec:obstructions}
\subsection{Bundles, representations, and analytic classes}
On a compact $X$ there is a diagonal inclusion
\[
 j:C(X)\oplus C(X)\longrightarrow A_{X,D}.
\]
Thus a pair of smooth bundles gives the explicit class
\begin{equation}\label{eq:diagonalclass}
 \kappa(E_+,E_-)=j_*([E_+],[E_-]).
\end{equation}
Place bundle projections in the diagonal blocks; smooth pairings are
legitimate by Lemma~\ref{lem:spectral-invariance}. This class forgets
the gluing on $U$. On noncompact $X$, an arbitrary bundle need not
define a compactly supported class. Comparisons between equivariant
bundles, representable $K$-theory, and analytic groupoid $K$-theory require
additional hypotheses \cite{EmersonMeyer2009,Tu2004,Moerdijk1995}.

\begin{Theorem}[Physical source-line obstruction]\label{thm:line-obstruction}
Let $r_L$ be restriction from $D_0=B\times\C$ to its zero section $L=B$.
For either sheet, the homomorphism
\[
 K_0(A_0)\xrightarrow{\rho_{\pm*}}K_c^0(D_0)
 \xrightarrow{r_L^*}K^0(L)
\]
is zero.  Hence every cyclic pairing which factors through either of these
restrictions to $L$, including their difference, vanishes on all of
$K_0(A_0)$.
\end{Theorem}
\begin{proof}
Let $b\in K_c^0(\C)$ be the complex Bott class.  The Thom isomorphism is
\[
 \beta:K^0(B)\overset\sim\longrightarrow K_c^0(B\times\C),
 \qquad \xi\longmapsto\xi\boxtimes b.
\]
Its restriction to the zero section multiplies $\xi$ by
$\lambda_{-1}(\C^*)=1-1=0$.  Therefore $r_L^*\beta=0$ and, since $\beta$
is onto, $r_L^*$ vanishes identically.  Naturality of cyclic pairings gives
the last assertion.
\end{proof}

A nonzero ordinary bundle degree on $L$ therefore cannot be this
analytic pairing. The supported construction below changes both the
class assignment and the cyclic current.

\subsection{The exact common-twist defect}
For a rank-$r$ pair on the compact model, write
$\Delta\ch_j=\ch_j(E_+)-\ch_j(E_-)$.  If $F$ is a line bundle with
$\ell=c_1(F)$, multiplicativity of the Chern character gives
\begin{equation}\label{eq:twist-law}
 \Delta\ch_1(E\otimes F)=\Delta\ch_1(E),\qquad
 \Delta\ch_2(E\otimes F)=\Delta\ch_2(E)+\ell\Delta\ch_1(E).
\end{equation}
Consequently the compact expression
\[
 Q_2^{\mathrm{comp}}(E_+,E_-)=-\frac14\int_{\overline D}\Delta\ch_2
\]
is not generally common-twist invariant.  For the asymmetric family
$E_+=\calO^r$, $E_-=\bigoplus_i M^{-2k_i}$, put
$S_j=\sum_i k_i^j$.  A twist by $M^a$ changes $Q_2^{\mathrm{comp}}$ from
$S_2$ to $S_2-aS_1$.  In rank one, $(\calO,M^{-2})$ gives $1$, whereas
$(M,M^{-1})$ gives $0$.  For the symmetric physical family
$E_\pm=\bigoplus_i M^{\pm k_i}$, $\Delta\ch_2=0$ identically.
These are ordinary characteristic-class calculations
\cite{BottTu1982,MilnorStasheff1974,Fulton1998}; Weyl invariance does not imply
common-twist invariance.

For completeness, the familiar discriminant correction
\[
 -\frac{1}{8r}\int_{\overline D}
 \bigl(\ch_2(\End E_+)-\ch_2(\End E_-)\bigr)
\]
is invariant under line twists because $\End(E\otimes F)\cong\End E$.
It equals $S_2-S_1^2/r$ on the asymmetric family, using
$\ch_2(\End E)=2r\ch_2(E)-c_1(E)^2$.  It still vanishes for the symmetric
pair and still uses the compact divisor.  It is therefore not the physical
replacement constructed in the following sections.

\section{Supported analytic source classes and cyclic pairings}
\label{sec:induction}
Let $D\subset X$ be smooth and $j_\Sigma:\Sigma\hookrightarrow D$ a compact
smooth embedded manifold with complex normal bundle $N$.  When $D$ is
oriented, orient $\Sigma$ using the orientation of $D$ and the complex normal
orientation.
The Thom class on a tubular neighbourhood, followed by extension by zero,
defines $j_{\Sigma!}:K^0(\Sigma)\to K_c^0(D)$ independently of the tubular
representative.

\begin{Definition}\label{def:source-induction}
The supported analytic source induction is
\begin{equation}\label{eq:source-induction}
 \Src_\Sigma(\xi)=\Theta_0^{-1}(0,j_{\Sigma!}\xi)
 \quad\text{in }K_0(A_{X,D}).
\end{equation}
In particular its divisor restrictions are
\begin{equation}\label{eq:source-restrictions}
 \rho_{+*}\Src_\Sigma(\xi)=j_{\Sigma!}\xi,
 \qquad \rho_{-*}\Src_\Sigma(\xi)=-j_{\Sigma!}\xi.
\end{equation}
\end{Definition}

This virtual class has total ambient class zero; existence and uniqueness
follow from Theorem~\ref{thm:K-splitting}, even when $j_{\Sigma!}\xi$ does
not extend from $X$.  Its distinction from a positive bundle class is
forced in the physical model by Lemma~\ref{lem:projectionless}.

\subsection{Cocycle conventions}
Suppose $D$ is oriented of real dimension $d$.  Let $\omega$ be a smooth
closed form of degree $d-2j$, with $j\geq1$.  On $C_c^\infty(D)$ define
\begin{equation}\label{eq:normalized-cocycle}
 \tau_{\omega,j}(a_0,\ldots,a_{2j})=
 \frac{1}{(2\pi i)^j(2j)!}
 \int_D\omega\wedge a_0\,da_1\wedge\cdots\wedge da_{2j}.
\end{equation}
Compact support of the arguments suffices; $\omega$ need not be compactly
supported.  Closedness and Stokes' theorem make this a cyclic cocycle.
We use the periodic Chern character convention with degree-$2j$ term
$(-1)^j(2j)!/j!\,\Tr((p-\tfrac12)\otimes p^{\otimes2j})$.
Thus the pairing normalization is
\begin{equation}\label{eq:pair-normalization}
 \langle\tau_{\omega,j},z\rangle
   =\int_D\omega\wedge\ch_j(z).
\end{equation}
For the Grassmann connection on the image of a projection \(p\), this is
\[
 \frac{(-1)^j}{j!(2\pi i)^j}
 \int_D \omega\wedge\Tr\bigl(p(dp)^{2j}\bigr).
\]
We use the convention
\(\ch(\calO(1))=\exp(H)\), where
\(\int_{\CP^1}H=1\).
The formula extends additively to virtual classes.
This fixes all signs and factorials, independently of alternative cocycle
conventions in \cite{Connes1985,Connes1994,Loday1998,Khalkhali2009}.
The matrix trace here is not divided by the matrix size.  For unitizations,
the cocycle is normalized relative to the scalar part.

Pull back along the two quotient maps and set
\[
 \tau_{\omega,j}^{\rel}=\rho_+^*\tau_{\omega,j}-\rho_-^*\tau_{\omega,j}.
\]
By Lemma~\ref{lem:spectral-invariance} these cocycles pair with the analytic
classes under discussion.  Changing $\omega$ by an exact form does not change
the resulting pairings.

\begin{Theorem}[Supported source formula]\label{thm:source-formula}
For every $\xi\in K^0(\Sigma)$,
\begin{equation}\label{eq:source-RR}
 \frac12\langle\tau_{\omega,j}^{\rel},\Src_\Sigma(\xi)\rangle
 =\int_\Sigma
 \left[j_\Sigma^*\omega\wedge\ch(\xi)\Td(N)^{-1}\right]_{[\dim\Sigma]}.
\end{equation}
The formula is independent of auxiliary Hermitian metrics, connections, and
tubular representatives, with the embedding and complex $K$-orientation fixed.
\end{Theorem}
\begin{proof}
By \eqref{eq:source-restrictions} and \eqref{eq:pair-normalization}, the left
side is $\int_D\omega\wedge\ch_j(j_{\Sigma!}\xi)$.  The Chern character of
the complex Thom class satisfies
\[
 \ch(j_{\Sigma!}\xi)=
 j_{\Sigma*}\bigl(\ch(\xi)\Td(N)^{-1}\bigr).
\]
For a normal line of Chern root $x$, the Euler-class comparison is
$(1-e^{-x})/x=\Td(x)^{-1}$; multiplicativity and the splitting principle
give the formula for $N$.  The projection formula now proves
\eqref{eq:source-RR}, the standard Thom/Riemann--Roch comparison
\cite{AtiyahKTheory1967,BottTu1982,Fulton1998} for the specified analytic
class.  Homotopy invariance and Chern--Weil theory give the independence
assertions.
\end{proof}

\subsection{The physical rank and degree pairings}
For $L=B\subset D_0=B\times\C$, write
$\beta(\xi)=\xi\boxtimes b$, with $\int_\C\ch_1(b)=1$, and abbreviate
$\Src=\Src_L$.  Choose a closed two-form $\eta$ on $B$ with
$\int_B\eta=1$. Define the individual-sheet cocycles on $\alg_{X_0,D_0}$ by
\begin{equation}\label{eq:individual-PsiPhi}
 \Psi^\pm=\rho_\pm^*\tau_{p^*\eta,1},\qquad
 \Phi^\pm=\rho_\pm^*\tau_{1,2}.
\end{equation}
Their relative differences are
\begin{equation}\label{eq:PsiPhi}
 \Psi^{\rel}=\Psi^+-\Psi^- =\tau_{p^*\eta,1}^{\rel},\qquad
 \Phi^{\rel}=\Phi^+-\Phi^- =\tau_{1,2}^{\rel}.
\end{equation}
\begin{Corollary}\label{cor:rank-degree}
For $\xi\in K^0(B)$,
\begin{equation}\label{eq:rank-degree}
 \frac12\langle\Psi^{\rel},\Src(\xi)\rangle=\rk\xi,
 \qquad
 \frac12\langle\Phi^{\rel},\Src(\xi)\rangle=\deg\xi.
\end{equation}
The map $\Src$ is injective, and these two pairings distinguish its image.
\end{Corollary}
\begin{proof}
The normal bundle of $L$ in $D_0$ has Todd class $1$.  Apply
Theorem~\ref{thm:source-formula} with $(\omega,j)=(p^*\eta,1)$ and
$(1,2)$.  Rank and degree identify $K^0(B)$ with $\Z^2$, proving both
injectivity and separation.
\end{proof}

Here $\Psi^{\rel}$ integrates on $D_0$ against $p^*\eta$, detecting the
normal Bott class.  It is therefore distinct from restriction to $L$;
\eqref{eq:rank-degree} uses neither that restriction nor integration over
$\overline D$.

\begin{Remark}[The position of the compact section]\label{rem:translated-section}
Replacing $L$ by $L_c=B\times\{c\}$, $c\in\C$, leaves the source class
and pairings unchanged: translation of the Bott representative is a homotopy,
and bundle restrictions to the sections agree in $K^0(B)$.  This independence
concerns section position in the specified product model, with the other
source and boundary data fixed.
\end{Remark}

\subsection{A complete integral pairing and the source coordinate}
\label{subsec:complete-pairing}
To determine the kernel of the relative pairings on the entire analytic
group, combine them with ambient integration.  For $V=M\oplus M$ and its
complex Thom isomorphism $\operatorname{Th}_V:K^0(B)\to K_c^0(X_0)$, write
\begin{equation}\label{eq:Thom-coordinates}
 \widehat\Theta=
 \bigl(\operatorname{Th}_V^{-1}\circ\iota_*,
                  \beta^{-1}\circ\rho_{+*}\bigr):
 K_0(A_0)\xrightarrow{\cong}K^0(B)\oplus K^0(B),
\end{equation}
and set
\[
 \widehat\Theta(z)=(x,a),\qquad
 x=x_0+x_1\epsilon,\qquad a=a_0+a_1\epsilon.
\]
Thus $x,a$ are inverse Thom images, rather than ordinary restrictions of
unspecified bundles to $B$. Equation~\eqref{eq:physical-restriction} gives
\begin{equation}\label{eq:physical-minus-coordinate}
 \beta^{-1}\rho_{-*}(z)=2\epsilon x-a
             =-a_0+(2x_0-a_1)\epsilon.
\end{equation}

View the arguments $b_j\in\alg_{X_0,D_0}$ as matrix-valued functions
through $\iota$, and define
\begin{align}
 \Upsilon_2(b_0,\ldots,b_4)
 &=\frac{1}{(2\pi i)^2 4!}\int_{X_0}p^*\eta\wedge
          \Tr(b_0\,db_1\wedge\cdots\wedge db_4),
                  \label{eq:ambient-four-cocycle}\\
 \Upsilon_3(b_0,\ldots,b_6)
 &=\frac{1}{(2\pi i)^3 6!}\int_{X_0}
          \Tr(b_0\,db_1\wedge\cdots\wedge db_6).
                  \label{eq:ambient-six-cocycle}
\end{align}
With the full matrix trace, compact supports, and conventions
\eqref{eq:normalized-cocycle}--\eqref{eq:pair-normalization}, these are
cyclic cocycles by Stokes' theorem.  They pair by integrating
$p^*\eta\wedge\ch_2(\iota_*z)$ and $\ch_3(\iota_*z)$, respectively.

\begin{Theorem}[Four integral cyclic coordinates]
\label{thm:complete-pairing}
For every $z\in K_0(A_0)$, in the coordinates
\eqref{eq:Thom-coordinates},
\begin{equation}\label{eq:four-pairings}
 \begin{aligned}
 \langle\Upsilon_2,z\rangle&=x_0,&
 \langle\Upsilon_3,z\rangle&=x_1-x_0,\\
 \tfrac12\langle\Psi^{\rel},z\rangle&=a_0,&
 \tfrac12\langle\Phi^{\rel},z\rangle&=a_1-x_0.
 \end{aligned}
\end{equation}
These four pairings define an isomorphism $K_0(A_0)\to\Z^4$.
In particular, the normalized half-relative pairings are integer-valued
on the entire analytic $K$-group.
\end{Theorem}
\begin{proof}
Let $H=c_1(M)$, with $\int_BH=1$, and let $u_X$ be the degree-four
cohomological Thom class of $V$. Since $H^2=0$ on $B$,
\[
 \Td(V)^{-1}=1-H,\qquad
 \ch(\operatorname{Th}_Vx)
 =u_X\wedge p^*\bigl(x_0+(x_1-x_0)H\bigr).
\]
Fibre integration gives the first two values in
\eqref{eq:four-pairings}. On $D_0=B\times\C$, the normal line of $B$
is trivial. Applying \eqref{eq:pair-normalization} to the Bott class
and using \eqref{eq:physical-minus-coordinate} gives
\begin{equation}\label{eq:individual-pairing-values}
 \begin{aligned}
 \langle\Psi^+,z\rangle&=a_0,&
 \langle\Psi^-,z\rangle&=-a_0,\\
 \langle\Phi^+,z\rangle&=a_1,&
 \langle\Phi^-,z\rangle&=2x_0-a_1.
 \end{aligned}
\end{equation}
Taking differences proves the remaining values. The pairing matrix in
coordinates $(x_0,x_1,a_0,a_1)$ is
\[
 \begin{pmatrix}
 1&0&0&0\\ -1&1&0&0\\ 0&0&1&0\\ -1&0&0&1
 \end{pmatrix},
\]
whose determinant is one. If the four pairing values are $(A,B,C,D)$,
the inverse is
\[
 (x_0,x_1,a_0,a_1)=(A,A+B,C,A+D).
\]
\end{proof}

Define the integral source coordinate
\begin{equation}\label{eq:source-coordinate}
 \mathcal R(z)=a-\epsilon x
              =a_0+(a_1-x_0)\epsilon\in K^0(B).
\end{equation}
Then the two normalized relative pairings are its rank and degree, and
$\mathcal R(\Src(\xi))=\xi$.

\begin{Corollary}[Sheet exchange and the supported subgroup]
\label{cor:source-eigenspaces}
Let $\sigma$ exchange the sheets. In Thom coordinates,
\[
 \widehat\Theta(\sigma_*z)=(x,2\epsilon x-a),
\]
and hence in coordinates $(x,\mathcal R)$ it acts by
\[
 (x,\mathcal R)\longmapsto(x,-\mathcal R).
\]
There is an integral direct-sum decomposition into its invariant and
anti-invariant subgroups. The supported induction has image exactly the
anti-invariant subgroup, and
\begin{equation}\label{eq:relative-pairing-kernel}
 \ker\bigl(\tfrac12\Psi^{\rel},\tfrac12\Phi^{\rel}\bigr)
 =\{\widehat\Theta^{-1}(x,\epsilon x):x\in K^0(B)\}
\end{equation}
is exactly the invariant subgroup. In this formula the kernel refers to
the homomorphisms on $K_0(A_0)$ supplied by the pairings.
\end{Corollary}
\begin{proof}
Sheet exchange is conjugation by the constant permutation matrix in the
ambient matrix algebra. It interchanges $\rho_+$ and $\rho_-$ and acts
trivially on the ambient $K$-class. Equation~\eqref{eq:physical-minus-coordinate}
gives the first formula. The integral change of coordinates
$(x,a)\mapsto(x,a-\epsilon x)$ has inverse
$(x,\xi)\mapsto(x,\epsilon x+\xi)$ and diagonalizes the involution.
Since $K^0(B)$ is torsion-free, its anti-invariant subgroup consists
precisely of the classes with $x=0$. These are the supported classes,
because $\widehat\Theta(\Src(\xi))=(0,\xi)$.
Theorem~\ref{thm:complete-pairing} proves the kernel assertion.
\end{proof}

\begin{Remark}[The entire divisor quotient]
\label{rem:divisor-quotient-kernel}
The two selected relative pairings detect two coordinates. The full
divisor quotient detects three: writing $\rho=(\rho_+,\rho_-)$,
\[
 \ker\rho_* =\Z\,\widehat\Theta^{-1}(\epsilon,0),
 \qquad
 \tfrac12\langle\Phi^++\Phi^-,z\rangle=x_0.
\]
Indeed, both divisor classes vanish exactly when $a=0$ and
$2\epsilon x=0$. The image is the rank-three subgroup with the integral
parity condition in Section~\ref{sec:K-theory}.
Every cyclic cocycle factoring through this quotient vanishes on the
displayed kernel class, which $\Upsilon_3$ detects.  This limitation is
specific to that factorization.
\end{Remark}

\subsection{The full periodic cyclic theory and its source information}
We use the geometric smooth cyclic theory with chain spaces
$C_n^{\mathrm{geo}}=C_c^\infty(\calG_{X,D}^{\,n+1})$, where the power is
Cartesian. Each chain space has its usual LF topology, and the tensor
convention is the completed inductive one
\cite[arXiv~v2, Section~1, p.~5]{CrainicMoerdijk2001}. The cyclic and
periodic complexes are understood via relative unitization;
equivalently, one may use the $(b,b')$ cyclic bicomplex, since this
\'etale convolution algebra is H-unital. These are the geometric chain
spaces of \cite[arXiv~v1, Section~4.1, pp.~23--24]{Crainic1999}.
Here $HP^\bullet_{\mathrm{cont}}$ is formed from their continuous duals
$(C_n^{\mathrm{geo}})'$; its comparison with the full complex-linear
duals $(C_n^{\mathrm{geo}})^\vee$ is justified below. The germ quotient
used to compute homology is a quotient of complexes of vector spaces,
without Hausdorff completion.

\begin{Theorem}[Complete periodic cyclic calculation]
\label{thm:full-periodic-cyclic}
For a second-countable Hausdorff smooth manifold $X$ without boundary
of finite dimension and a closed smooth submanifold $D$ without boundary, one has
\[
 HP_j(\alg_{X,D})\cong
 \bigoplus_{k\equiv j\; (\mathrm{mod}\,2)}
 \bigl(H_c^k(X;\C)\oplus H_c^k(D;\C)\bigr),\qquad j=0,1.
\]
In the physical model,
\[
 HP_0(\alg_{X_0,D_0})\cong\C^4,\qquad
 HP_1(\alg_{X_0,D_0})=0,
\]
and the corresponding continuous periodic cohomology has the same
dimensions. The Chern--Connes character
$K_0(A_0)\otimes\C\to HP_0(\alg_{X_0,D_0})$ is an isomorphism, and
\[
 [\Upsilon_2],\quad[\Upsilon_3],\quad
 \tfrac12[\Psi^{\rel}],\quad\tfrac12[\Phi^{\rel}]
\]
form a basis of its even periodic cyclic cohomology.
\end{Theorem}
\begin{proof}
The groupoid is principal, so its loop space consists only of units.
Crainic's unit-localization theorem
\cite[arXiv~v1, Theorem~4.3, p.~26]{Crainic1999} expresses its full
periodic homology as the parity products of $H_k(\calG_{X,D};\C)$.
These groups are computed by the compactly supported bar complex of the
$c$-soft de Rham resolution
\cite[arXiv~v1, \S\S~3.1--3.5, pp.~9--11]{CrainicMoerdijk2000},
with total homological degree $p-q$ for groupoid degree $p$ and form
degree $q$.
Write $j:U=X\setminus D\hookrightarrow X$ and let $j_!$ denote
extension by zero of compactly supported forms. In normalized groupoid
degree $p$, the row of $q$-forms is
\[
 B_0^q=\Omega_c^q(X)^{\oplus2},\qquad
 B_p^q=\Omega_c^q(U)^{\oplus2}\quad(p\geq1).
\]
Indeed, a nondegenerate string has alternating sheet labels; composing
two consecutive cross-sheet arrows produces a degenerate internal face.
With $S(\alpha,\beta)=(\beta,\alpha)$, the boundary for $p\geq2$ is
$S+(-1)^p\id$, and the first boundary is
$(\alpha,\beta)\mapsto(j_!(\beta-\alpha),j_!(\alpha-\beta))$.
The kernels and images alternate between the diagonal and anti-diagonal
subspaces. Thus every positive groupoid degree is exact and degree zero is
\[
 Q^\bullet=
 \frac{\Omega_c^\bullet(X)^{\oplus2}}
      {\{(j_!\eta,-j_!\eta):\eta\in\Omega_c^\bullet(U)\}}
 \cong\Omega_c^\bullet(X)\oplus
       \frac{\Omega_c^\bullet(X)}{j_!\Omega_c^\bullet(U)},
\]
where the last map is $[(\alpha,\beta)]\mapsto
(\alpha+\beta,[\alpha])$. The last quotient is the complex of compactly supported germs of
ambient forms along $D$. Its restriction to $\Omega_c^\bullet(D)$ is a
quasi-isomorphism. Indeed, on a tubular neighbourhood with retraction
$\pi_D$, the radial normal homotopy $K$ gives
\[
 \alpha-\pi_D^*i^*\alpha=dK\alpha+Kd\alpha.
\]
For compactly supported germs, choose a cutoff equal to one near their
compact footprint on $D$ and with compact support in the tubular
neighbourhood. It extends the pulled-back representatives and homotopies
to compactly supported ambient forms. Terms containing a cutoff
derivative vanish as germs: the derivative vanishes near the compact
footprint, while the relevant germs vanish outside that footprint.
Thus the identity holds in the quotient. This proves both
surjectivity and injectivity on cohomology. Restriction is a
quasi-isomorphism, not an identification of the complexes.
After the positive groupoid degrees have vanished, a surviving form
degree $q$ has homological degree $-q$, of the same parity as $q$;
bounded form degree makes the periodic parity product finite. This gives
the displayed homology formula without identifying the smooth algebra
quotient with $C_c^\infty(D)^{\oplus2}$.

For $X_0=\Tot(\calO(1)^{\oplus2})$, compactly supported cohomology
occurs in degrees $4,6$; for $D_0=\CP^1\times\C$, in degrees $2,4$.
Each of these groups is one-dimensional.

For cohomology, put $\calG=\calG_{X_0,D_0}$ and $U=X_0\setminus D_0$.
Its arrow manifold
$X_{0,+}\sqcup X_{0,-}\sqcup U_{+-}\sqcup U_{-+}$ is Hausdorff.
For a Hausdorff groupoid, the current comparison in
\cite[arXiv~v1, Section~4.10 and the paragraph following
Definition~4.11, p.~29]{Crainic1999} states that inclusion of the
continuous duals into the full duals of these geometric chains induces
an isomorphism on periodic cohomology. This cochain inclusion preserves
evaluation on cycles. The complete loop space is
$O=\calG^{(0)}=X_{0,+}\sqcup X_{0,-}$: it is invariant, open in the loop
space, elliptic because all its arrows are identities, and a smooth
real six-manifold. Thus every hypothesis of
\cite[arXiv~v1, Theorem~4.12, pp.~29--30]{Crainic1999} holds, and
localization at $O$ is the whole theory. Combining its injective
pairing with the preceding comparison gives
\begin{equation}\label{eq:periodic-evaluation-injective}
 HP^j_{\mathrm{cont}}(\alg_{X_0,D_0})
 \lhook\joinrel\longrightarrow
 \Hom_\C\bigl(HP_j(\alg_{X_0,D_0}),\C\bigr),
 \qquad j=0,1.
\end{equation}
This uses the comparison on the original geometric chains, before the
germ quotient; it requires no closed-range assertion for that quotient.
The even cohomology therefore has dimension at most four and the odd
cohomology vanishes. The four continuous cocycles in
Theorem~\ref{thm:complete-pairing} are independent by their unimodular
pairing matrix, so they form a basis. The same matrix and compatibility
of the idempotent Chern character with cyclic evaluation make
$K_0(A_0)\otimes\C\to HP_0(\alg_{X_0,D_0})$ injective; the
four-dimensional target makes it an isomorphism.
Lemma~\ref{lem:spectral-invariance} supplies representatives in the
unitized smooth core.
\end{proof}

\begin{Corollary}[Exhaustion of stable cyclic source observables]
\label{cor:cyclic-source-exhaustion}
For every even periodic cyclic class $[\tau]$ in the theory just
specified there are unique constants $a,b,c,d\in\C$ such that
\[
 [\tau]=a[\Upsilon_2]+b[\Upsilon_3]
       +\tfrac c2[\Psi^{\rel}]+\tfrac d2[\Phi^{\rel}].
\]
Consequently, on supported source classes,
\[
 \langle[\tau],\Src(\xi)\rangle
       =c\,\rk\xi+d\,\deg\xi.
\]
Thus the complete periodic cyclic theory supplies no additional stable
source observable beyond rank and degree in this model.
\end{Corollary}
\begin{proof}
Apply the basis statement and $\iota_*\Src(\xi)=0$, followed by
Corollary~\ref{cor:rank-degree}.
\end{proof}

The factorization through rank and degree follows already from
$K^0(B)\cong\Z^2$. The additional statement proved here is completeness:
there are no further periodic cyclic classes in this specified theory
outside the four displayed coordinates. This concerns stable source
classes and their periodic pairings; extension and connection data will
be tested in Section~\ref{sec:mixed-lines}.

\section{Relative normal residues of holomorphic modifications}
\label{sec:residue}
Relative $K$-theory retains stable normal gluing data forgotten by
\eqref{eq:diagonalclass}.  Let $D\subset X$ be a closed real-codimension-two
embedded submanifold of a smooth manifold, with specified complex normal
line $N_{D/X}$ (the holomorphic normal line for a complex divisor).
For a compact smooth embedded $\Sigma\subset D$, set
$N=N_{D/X}|_\Sigma$; using $\Src_\Sigma$ additionally requires the complex
normal orientation of $N_{\Sigma/D}$ from Section~\ref{sec:induction}.

Let $\alpha:E_-|_{X\setminus D}\to E_+|_{X\setminus D}$ be an
isomorphism of smooth complex bundles, meromorphic in the holomorphic
examples.  Pull it back to a small closed normal disk bundle over $\Sigma$
using a tubular embedding inducing the identity on $D$ and its normal
bundle.  Its boundary is compact and disjoint from $D$, so $\alpha$ is
invertible there and admits the polar homotopy
$\alpha_s=\alpha(\alpha^*\alpha)^{-s/2}$ to a unitary map.  Thus no
meromorphic coefficient at $D$ must act as a bounded groupoid-algebra
operator.

For $\pi:N\to\Sigma$, let $z$ be the tautological section of $\pi^*N$.
Our relative-triple convention is $(E_+,E_-,\alpha:E_-\to E_+)$, with
forgetful class $[E_+]-[E_-]$.  The complex Thom class is the Koszul triple
\[
 u_N=\bigl(\underline{\C},\pi^*N^*,\operatorname{ev}_z\bigr),
 \qquad
 \operatorname{ev}_z:\pi^*N^*\longrightarrow\underline{\C},
 \quad \eta\longmapsto\eta(z),
\]
where the evaluation is invertible on the circle bundle.  Its restriction
to the zero section is $1-[N^*]$, in agreement with
\eqref{eq:physical-restriction} and Theorem~\ref{thm:source-formula}.
For any smooth complex bundle $E\to\Sigma$, the one-step multiplication
modification $(\pi^*(E\otimes N),\pi^*E,z)$ is $[E\otimes N]u_N$, and its inverse
Thom image is $[E\otimes N]$, not $[E]$.
Relative triples and their mapping-cone formulations are standard
\cite{AtiyahKTheory1967,Haslehurst2021}.

\begin{Definition}\label{def:normal-residue}
The normal residue $\Res_\Sigma(E_+,E_-,\alpha)\in K^0(\Sigma)$ is the
inverse Thom image of this relative class.  Its supported analytic image is
$\Src_\Sigma(\Res_\Sigma)$.
\end{Definition}

The residue uses $N_{D/X}|_\Sigma$ for transverse gluing, whereas induction
uses $N_{\Sigma/D}$ for compact support inside $D$: these are $M^2$ and
the trivial line in the physical model.  The two Thom operations enter the
specified composition
\[
 (E_+,E_-,\alpha)\longmapsto\Res_\Sigma\in K^0(\Sigma)
 \longmapsto\Src_\Sigma(\Res_\Sigma)\in K_0(A_{X,D}).
\]

\begin{Proposition}\label{prop:residue-invariance}
The residue is invariant under isomorphisms of modification triples, homotopies
through boundary-invertible gluing maps, and changes of the normal disk
radius or compatible tubular embedding, with the complex normal orientation
fixed.  It is additive under direct sum.  Reversing a modification negates
the residue.  Tensoring both bundles by a line bundle $F$ tensors the residue
by $F|_\Sigma$.
\end{Proposition}
\begin{proof}
Compatible tubular embeddings become isotopic after shrinking near compact
$\Sigma$, giving a homotopy of relative triples with boundary disjoint from
$D$.  Metrics and radii vary similarly.  Isomorphism invariance, homotopy
invariance and additivity hold in relative $K$-theory; reversal interchanges
the bundles and inverts the boundary map, giving the negative class.
The inverse Thom map is a $K^0(\Sigma)$-module map, proving the tensor rule.
\end{proof}

This residue uses the gluing map, but only up to stable topological homotopy.
It is not a complete invariant of a holomorphic modification.

\begin{Theorem}[Coulomb residue and charge]\label{thm:Coulomb-residue}
For the symmetric data \eqref{eq:symmetric-coulomb} and $\Sigma=L=B$,
\begin{equation}\label{eq:residue-Coulomb}
 \Res_L(E_{n,+},E_{n,-},\alpha_n)=n[M]\quad\text{in }K^0(B).
\end{equation}
Consequently the analytic class $\Src(n[M])\in K_0(A_0)$ satisfies
\begin{equation}\label{eq:Coulomb-supported-charge}
 \frac12\langle\Psi^{\rel},\Src(n[M])\rangle=n.
\end{equation}
\end{Theorem}
\begin{proof}
For $n\geq1$, multiplication by $q^n$ has a filtration by successive powers
of the divisor ideal.  Equivalently, factor it into $n$ multiplication-by-$q$
steps and use additivity of relative triples.  The normal residue is
\begin{equation}\label{eq:residue-finite-sum}
 \sum_{j=0}^{n-1}[M^{n-2j}|_B].
\end{equation}
Restriction to the zero section verifies the normalization: multiplication by the
Euler factor $1-M^{-2}$ sends this sum to $M^n-M^{-n}$.  Each summand itself
comes from the corresponding one-step Koszul/Thom class, so the formula
is not obtained by dividing by a possibly noninvertible Euler class.
In $K^0(B)=\Z[\epsilon]/(\epsilon^2)$, $[M^k]=1+k\epsilon$ for every integer $k$.
The sum of the exponents $n,n-2,\ldots,2-n$ is $n$, so
\eqref{eq:residue-finite-sum} equals $n+n\epsilon=n[M]$.
For negative $n$, reverse the positive modification; $n=0$ is trivial.
Equation~\eqref{eq:Coulomb-supported-charge} follows from
Corollary~\ref{cor:rank-degree}.
\end{proof}

The normalization in \eqref{eq:Coulomb-supported-charge} follows from the
ordered-sheet difference and the complex Bott orientation.  It is not a
claim that the original holomorphic line bundle itself is a compactly
supported analytic class.

\begin{Corollary}[What the stable residue loses]\label{cor:stable-loss}
Let $\rho:G\to GL(V_\rho)$ be a representation of a connected complex
reductive group and let $\lambda:\C^*\to T\subset G$ be a cocharacter.
Put $k_\mu=\langle\mu,\lambda\rangle$ and $m_\mu=\dim V_\mu$.
For the associated symmetric toral modification,
\[
 \Res_{L,\rho}=\Bigl(\sum_\mu m_\mu k_\mu\Bigr)[M].
\]
If $\det\rho$ is trivial, this residue and all pairings with its supported
analytic image vanish.
\end{Corollary}
\begin{proof}
Apply Theorem~\ref{thm:Coulomb-residue} to the weight summands and add.
The sum of the weights with multiplicity is the determinant character.
\end{proof}

For this symmetric Coulomb family over $B$, a nonzero semisimple quadratic
moment therefore requires data beyond the stable residue.  The next section
retains such data.  The vanishing assertion is specific to this family and
source cycle.

\subsection{Higher-dimensional source cycles}
The normal residue has an explicit extension before the relation $\epsilon^2=0$
on $B$ is imposed.
\begin{Proposition}\label{prop:higher-residue}
Suppose a smooth divisor $D\subset X$ is defined by a section $q$ of $M^2$,
where $M$ is a line bundle.  For a compact cycle $\Sigma\subset D$ and the
symmetric modification $(M^n,M^{-n},q^n)$, let $x=c_1(M|_\Sigma)$.  Then
\[
 \Res_{\Sigma,n}=\sum_{a=0}^{n-1}[M^{n-2a}|_\Sigma]\quad(n>0),
 \qquad\Res_{\Sigma,-n}=-\Res_{\Sigma,n},
\]
and its Chern character is the finite-degree truncation of the formal series
\begin{equation}\label{eq:residue-hyperbolic}
 \ch(\Res_{\Sigma,n})=e^x\frac{\sinh(nx)}{\sinh x}
 =n+nx+\frac{n(n^2+2)}{6}x^2+\frac{n^3}{6}x^3+\cdots.
\end{equation}
The quotient in this formula is interpreted as its regular power series at
$x=0$, not as division by a cohomology class.
\end{Proposition}
\begin{proof}
The successive Koszul classes used in
\eqref{eq:residue-finite-sum} apply on any such $\Sigma$, without the
$K$-theory relation of $B$.  Their Chern characters sum to
$\sum_{a=0}^{n-1}e^{(n-2a)x}=e^x\sinh(nx)/\sinh x$.
Expanding this finite exponential sum gives \eqref{eq:residue-hyperbolic}.
Reversal proves the negative case, and the zero case is zero.
\end{proof}

For $S_j=\sum_\mu m_\mu k_\mu^j$, the degree-four term is
$(S_3+2S_1)x^2/6$, which can survive when $S_1=0$ on a larger source cycle.
On $B$, $x^2=0$ gives Corollary~\ref{cor:stable-loss}.
Theorem~\ref{thm:source-formula} supplies the supported pairings with their
normal Todd correction; no identification with physical multipoles is
asserted.

\section{Principal bundles and intrinsic normal filtrations}
\label{sec:filtered}
Throughout this section, \(X\) is a complex manifold, \(D\subset X\) is a
smooth effective Cartier divisor, and the modifications are holomorphic
vector-bundle modifications with meromorphic gluing along \(D\).
Principal groups are connected complex reductive algebraic groups, and
their finite-dimensional representations are algebraic.

\subsection{Principal toral modifications and orbit types}
Let $G$ be connected and complex reductive, $T\subset G$ a maximal torus,
and $\lambda:\C^*\to T$ a cocharacter.  Extend the structure group of the
Coulomb line bundle by $\lambda$.  Its principal bundles on the two sheets
are obtained from the frame bundles of $M$ and $M^{-1}$ by extension of
structure group.  The gluing is induced by the bundle isomorphism $q^{-1}$
on $U$.  In a representation $\rho$, the associated bundles and the reverse
gluing are
\begin{equation}\label{eq:principal-symmetric}
 E_{\rho,+}=\bigoplus_\mu M^{k_\mu}\otimes V_\mu,\qquad
 E_{\rho,-}=\bigoplus_\mu M^{-k_\mu}\otimes V_\mu,\qquad
 \alpha_\rho=\bigoplus_\mu q^{k_\mu}.
\end{equation}
These are the cocharacter constructions already present in twistorial
line-defect theory \cite{GarnerPaquette2025}; their magnetic-charge
interpretation is consistent with \cite{GoddardNuytsOlive1977,KapustinWitten2007}.

Choose a local normal coordinate $z$ for $D$.  With the direction of
$\alpha$ used here, the normal lattice is $z^\lambda G(\C[[z]])$.
Reversing the gluing direction changes $\lambda$ to $-\lambda$.
In $\Gr_G=G(\C((z)))/G(\C[[z]])$, a dominant cocharacter (equivalently a
Weyl orbit) indexes a $G(\C[[z]])$-orbit; individual points retain further
flag and jet information.  Along $D$, coefficients vary and normal coordinates
transform with the normal line, so the construction is a family rather
than a single scalar Laurent-series problem
\cite{BeauvilleLaszlo1995,PressleySegal1986,MirkovicVilonen2007,Zhu2016}.

\subsection{The filtered construction and its precise hypothesis}
Let $\mathscr I=\calO_X(-D)$ and let
$\mathscr L_+=E_+$, $\mathscr L_-=\alpha(E_-)$ be lattices in the common
meromorphic bundle $E_+(*D)$.  The following hypothesis specifies the class
of modifications for which we use associated graded bundles.

\begin{Definition}\label{def:constant-type}
A modification has \emph{locally split constant normal type} near $\Sigma$
if, in formal coordinates along $D$ and local bundle frames, it has the form
\[
 \mathscr L_+=\calO_D[[z]]^r,\qquad
 \mathscr L_-=a(z)\diag(z^{k_1},\ldots,z^{k_r})\calO_D[[z]]^r,
\]
where $a(z)$ is formally invertible and the multiset of integers $k_i$ is
locally constant.  All bundles and graded pieces below are restricted to a
neighbourhood of $\Sigma$ where this hypothesis holds.
\end{Definition}

These are the completions of the given holomorphic lattices.  The toral
families \eqref{eq:principal-symmetric} and elementary modifications of
Section~\ref{sec:family} satisfy this regularity hypothesis.

Define an increasing filtration of $V=E_+|_D$ by
\begin{equation}\label{eq:normal-filtration}
 F_mV=\im\bigl((\mathscr I^{-m}\mathscr L_-\cap\mathscr L_+)
                \longrightarrow\mathscr L_+/\mathscr I\mathscr L_+\bigr),
 \qquad W_m=F_mV/F_{m-1}V.
\end{equation}
The ideal-sheaf definition is coordinate independent.  Locally it is the
lattice-to-flag projection $\Gr_\lambda\to G/P_\lambda$
\cite[projection $p_\mu$ before Example~2.1.12, and Example~2.1.12,
arXiv~v2, pp.~26--27]{Zhu2016}, varying along $D$;
we record its graded bundles in $K^0(\Sigma)$.

A common line twist tensors both sheetwise bundles by the same line bundle
$F$ on $X$, identified by the identity off $D$; its normal exponent is zero.

\begin{Proposition}[Filtered profile]\label{thm:profile}
Under Definition~\ref{def:constant-type}, the $F_mV$ and $W_m$ are vector
bundles, only finitely many $W_m$ are nonzero, and
\begin{equation}\label{eq:profile}
 \Xi_\Sigma(E_+,E_-,\alpha;t)=\sum_{m\in\Z}[W_m|_\Sigma]t^m
 \quad\text{belongs to }K^0(\Sigma)[t,t^{-1}].
\end{equation}
It is invariant under isomorphisms of modifications and changes of normal
coordinate.  Tensoring by a common line bundle $F$ multiplies every
coefficient by $[F|_\Sigma]$.  For a representation of a principal
type-$\lambda$ modification,
\[
 \rk W_m=\sum_{\langle\mu,\lambda\rangle=m}\dim V_\mu.
\]
\end{Proposition}
\begin{proof}
In the displayed local form, $\mathscr I^{-m}\mathscr L_-$ is generated
by $a(z)z^{k_i-m}e_i$.  Because $a(z)$ preserves $\mathscr L_+$, reduction
of its intersection with $\mathscr L_+$ is the span of $a(0)e_i$ with
$k_i\leq m$.  Thus $F_m$ is a locally free subbundle of constant rank and
$W_m$ has one generator for each $k_i=m$.  The intrinsic definition
\eqref{eq:normal-filtration} glues these local subbundles and proves
coordinate and isomorphism invariance.  A common line tensor factor commutes
with intersections of these locally free lattices and their quotients.
For an associated representation, the exponents are exactly
$\langle\mu,\lambda\rangle$, with the weight multiplicities.
\end{proof}

The profile determines the stable residue (Theorem~\ref{thm:profile-residue})
but discards higher jets: the distinct framed lattices generated by
$e_1+cze_2,z^2e_2$ in $\C[[z]]^2$ have the same filtration at $z=0$
for every $c$.  Thus it does not determine the affine-Grassmannian point.

\begin{Proposition}[Tensor structure of the profile]\label{prop:tensor-profile}
For modifications satisfying Definition~\ref{def:constant-type}, the profile
is additive under direct sum and multiplicative under tensor product:
\[
 \Xi(E\oplus E';t)=\Xi(E;t)+\Xi(E';t),\qquad
 \Xi(E\otimes E';t)=\Xi(E;t)\Xi(E';t).
\]
For a fixed principal $G$-modification locally of the form
$a(z)z^\lambda b(z)$ with $a,b\in G(\calO_D[[z]])$, associated bundles define
a ring homomorphism
\begin{equation}\label{eq:representation-profile}
 R(G)\longrightarrow K^0(\Sigma)[t,t^{-1}],\qquad
 [\rho]\longmapsto\Xi_{\Sigma,\rho}(t).
\end{equation}
\end{Proposition}
\begin{proof}
Direct sums commute with the intrinsic filtration.  For tensor products,
local diagonal forms have exponents $k_i+k'_l$ and leading frame
$a(0)\otimes a'(0)$.  It follows that the filtration is the tensor filtration
and that its associated graded bundle has the canonical decomposition
\[
 W_m(E\otimes E')\cong
 \bigoplus_{u+v=m}W_u(E)\otimes W_v(E').
\]
This identity is intrinsic and therefore glues across changes of coordinates
and frames.  It proves multiplicativity.  Associated bundles respect sums
and tensor products.  Finite-dimensional representations of a complex
reductive group are completely reducible, so their Grothendieck relations are
generated by these direct-sum relations.  The trivial representation maps to
$[\calO_\Sigma]t^0$, proving \eqref{eq:representation-profile}.
\end{proof}

These tensor and graded operations are those of filtered fibre functors
\cite[Section~4.1]{Ziegler2015}.  The coefficient rank map recovers the
usual weight character; before applying it, the profile also retains bundle
classes along $\Sigma$.  Supported analytic induction is additive on these
coefficients and supplies no multiplicative structure on $K_0(A_{X,D})$.

\subsection{Recovering the stable residue from the filtered profile}
Under Definition~\ref{def:constant-type}, the relation between the two
constructions is explicit.  Write $N=N_{D/X}|_\Sigma$ and define the following
finite Laurent polynomials in its class:
\begin{equation}\label{eq:normal-factors}
 B_m(N)=
 \begin{cases}
  \displaystyle\sum_{a=0}^{m-1}[N^{-a}],&m>0,\\[4pt]
  0,&m=0,\\[4pt]
  \displaystyle-\sum_{a=1}^{-m}[N^a],&m<0.
 \end{cases}
\end{equation}
They obey $(1-[N^{-1}])B_m(N)=1-[N^{-m}]$.  This identity involves no
inversion of the Euler class.

\begin{Theorem}[The profile determines the stable normal residue]
\label{thm:profile-residue}
For a modification satisfying Definition~\ref{def:constant-type} near
$\Sigma$, in the Thom convention of Section~\ref{sec:residue},
\begin{equation}\label{eq:residue-from-profile}
 \Res_\Sigma(E_+,E_-,\alpha)
   =\sum_{m\in\Z}[W_m|_\Sigma]B_m(N)
       \quad\text{in }K^0(\Sigma).
\end{equation}
Consequently the profile, together with the normal line, determines both the
stable residue and its supported analytic image.  The statement does not
assert that the residue determines the profile.
\end{Theorem}
\begin{proof}
The zero-rank case is empty.  Otherwise let $m_{\min}$ and $m_{\max}$
be the least and greatest integers with $W_m|_\Sigma\ne0$.
Compactness and the constant-type local form permit
$\ell\leq\min(0,m_{\min})$ and $b>\max(0,m_{\max})$ such that near $\Sigma$
\[
 \mathscr I^b\mathscr L_+\subset\mathscr L_-
                 \subset\mathscr I^{\ell}\mathscr L_+.
\]
Coherence and faithful flatness of completion at points of $D$ justify
checking these inclusions formally.  The quotients lie on a finite thickening,
so their completions determine the finite filtrations below.  Constant normal
type makes the graded quotients locally free over $\calO_D$; convergence of
an infinite formal splitting is unnecessary.
Set $P=\mathscr I^b\mathscr L_+$ and $Q_\pm=\mathscr L_\pm/P$.
The original relative class is the difference of the classes of the two
inclusions $P\hookrightarrow\mathscr L_+$ and
$P\hookrightarrow\mathscr L_-$.  Thus it is represented by
$[Q_+]-[Q_-]$, with support on the divisor.

The order filtration of $Q_+$ has graded bundles
$N_{D/X}^{-a}\otimes V$ for $0\leq a<b$.  For $Q_-$ use the descending
filtration
\[
 J_a=(\mathscr L_-\cap\mathscr I^a\mathscr L_+)/P,
       \qquad \ell\leq a\leq b.
\]
The intrinsic definition \eqref{eq:normal-filtration} identifies its
successive quotient as
\[
 J_a/J_{a+1}\cong N_{D/X}^{-a}\otimes F_aV.
\]
Indeed multiplication by the fractional ideal $\mathscr I^a$ identifies
its image in
$\mathscr I^a\mathscr L_+/\mathscr I^{a+1}\mathscr L_+$ with
$N_{D/X}^{-a}\otimes F_aV$; its kernel is exactly $J_{a+1}$.
All these quotients are bundles on $D$ by the local splitting hypothesis.

Additivity for these finite filtrations, followed by the one-step Koszul
identification of a bundle supported on $D$ with its Thom class, gives
\[
 \Res_\Sigma
  =\sum_{a=0}^{b-1}[N^{-a}][V|_\Sigma]
   -\sum_{a=\ell}^{b-1}[N^{-a}][F_aV|_\Sigma].
\]
This is the usual d\'evissage of a finite normal modification, here restricted
to the normal disk over $\Sigma$; it can equivalently be performed with
relative bundle complexes \cite{AtiyahKTheory1967,Fulton1998}.
Since $[F_aV]=\sum_{m\leq a}[W_m]$ and $[V]=\sum_m[W_m]$, the coefficient
of $[W_m|_\Sigma]$ is
\[
 \sum_{a=0}^{b-1}[N^{-a}]-\sum_{a=m}^{b-1}[N^{-a}]=B_m(N).
\]
This proves \eqref{eq:residue-from-profile} without dividing by an Euler
class or assuming a global splitting of the lattices.  Applying the additive
map $\Src_\Sigma$ proves the last assertion.
\end{proof}

For the Coulomb family, $W_n|_B=M^n$ and $N=M^2$, so the theorem recovers
\eqref{eq:residue-finite-sum}.  For the elementary $GL_2$ family of
Section~\ref{sec:family}, it gives $\Res_B=\calO_B(d)$; for its $SL_2$
family, it gives $\Res_B=[\calO_B(d)]-[\calO_B(2-d)]$.
These examples apply the same formula; higher normal jets can still
distinguish modifications with the same stable residue.

\section{Cyclic moments, residues, and finite profile reconstruction}
\label{sec:moments}
Return to $\Sigma=L=B\subset D_0$.  For a modification satisfying
Definition~\ref{def:constant-type}, define
\begin{equation}\label{eq:profileclasses}
 C_m=\Src([W_m|_B]),\qquad
 C_j^{\wt}=\sum_m m^j C_m\quad(j\geq0),
\end{equation}
with $m^0=1$ also at $m=0$.  These finite sums are induced virtual
analytic classes, distinct from an asserted class of the original Ward bundle.

\begin{Proposition}[Cyclic realization of the weight polynomial]\label{thm:moments}
The cyclic rank polynomial is
\begin{equation}\label{eq:weight-polynomial}
 P_\rho(t)=\sum_m\frac12\langle\Psi^{\rel},C_m\rangle t^m
          =\sum_\mu m_\mu t^{\langle\mu,\lambda\rangle}.
\end{equation}
In particular
\begin{equation}\label{eq:corrected-moments}
 \widehat Q_{j,\rho}(\lambda)
   :=\frac12\langle\Psi^{\rel},C_j^{\wt}\rangle
   =\sum_\mu m_\mu\langle\mu,\lambda\rangle^j.
\end{equation}
All these numbers are invariant under common line-bundle twists, changes of
normal coordinate, isomorphisms of modifications, and the Weyl action on
$\lambda$.  They require no compactification of $D_0$.
\end{Proposition}
\begin{proof}
Corollary~\ref{cor:rank-degree} turns the pairing with $C_m$ into $\rk W_m$.
Proposition~\ref{thm:profile} identifies that rank with the weight multiplicity.
A line twist changes $W_m$ to $W_m\otimes F$, preserving its rank.
Coordinate and isomorphism invariance hold for the filtration, and the Weyl
group permutes representation weights with multiplicity.  The classes and
cocycles are constructed entirely on $X_0,D_0,B$.
\end{proof}

The quadratic moment $j=2$ is an invariant of the specified filtered type.
Its input $C_2^{\wt}$ is formed from the normal graded pieces, whereas
the compact expression $-\tfrac14\int_{\overline D}\Delta\ch_2$ uses
sheetwise Chern characters. It also differs from the unfiltered residue,
so Corollary~\ref{cor:stable-loss} does not force it to vanish.

\begin{Corollary}[Recovery of orbit type]\label{cor:reconstruction}
For $G=GL_r$ in the defining representation, $P_\rho$ recovers the unordered
cocharacter $(k_1,\ldots,k_r)$.  Equivalently, the rank and the first $r$
moments recover it.  For a connected complex reductive $G$, the polynomials
$P_\rho$ over all finite-dimensional representations determine the
cocharacter up to the Weyl group.
\end{Corollary}
\begin{proof}
In the $GL_r$ case the coefficient of $t^m$ is the multiplicity of $m$.
Alternatively, Newton's identities recover the elementary symmetric
polynomials from the first $r$ power sums, and hence the monic polynomial
with roots $k_1,\ldots,k_r$.  For general $G$, the characters of irreducible
representations span the Weyl-invariant character algebra of the torus; this
is the standard highest-weight character description
\cite{FultonHarris1991}.  If all $P_\rho$ agree for $\lambda$ and $\nu$,
then all characters agree on $\lambda(s)$ and $\nu(s)$ for any fixed real
$s>1$.  Characters separate semisimple conjugacy classes, so
$\lambda(s)=w\nu(s)$ for some $w$ in the Weyl group.  Evaluating each torus
character gives
$s^{\langle\chi,\lambda-w\nu\rangle}=1$; the exponent is an integer and
$s>1$, hence all exponents vanish.  Thus $\lambda=w\nu$.
\end{proof}

This is a cyclic realization of standard representation-theoretic recovery,
not a new classification of affine-Grassmannian orbits.  The orbit
classification itself is the one in \cite{MirkovicVilonen2007,Zhu2016}.

\subsection{Examples and loss of information in a single scalar}
For $G=\C^*$ and charge $n$, $P(t)=t^n$ and
$\widehat Q_1=n$, $\widehat Q_2=n^2$.  For the defining representation of
$GL_r$,
\[
 P(t)=\sum_{i=1}^r t^{k_i},\qquad
 \widehat Q_1=\sum_i k_i,\qquad \widehat Q_2=\sum_i k_i^2.
\]
For $SL_r$, the first moment vanishes but the second is positive for every
nonzero cocharacter in the defining representation.  In the adjoint
representation of $SL_r$,
\[
 \widehat Q_{2,\mathrm{Ad}}=
 \sum_{i\ne j}(k_i-k_j)^2=2r\sum_i k_i^2
 \quad\text{when }\sum_i k_i=0.
\]
For $SL_2$ with cocharacter $(n,-n)$, the adjoint value is $8n^2$.

A single quadratic moment is not a complete source invariant.  For example,
the $SL_3$ cocharacters $(1,1,-2)$ and $(-1,-1,2)$ both have quadratic moment
$6$ but are not permutations of one another.  Their cubic moments are
$-6$ and $6$, and their weight polynomials distinguish them.

\subsection{Degree data and a twist-invariant centred pairing}
Define, for a fixed representation,
\begin{equation}\label{eq:degree-moments}
 R_j=\frac12\langle\Psi^{\rel},C_j^{\wt}\rangle,
 \qquad H_j=\frac12\langle\Phi^{\rel},C_j^{\wt}\rangle
          =\sum_m m^j\deg(W_m|_B).
\end{equation}
Here $R_0=r>0$.  A common line twist of degree $a$ changes $H_j$ to
$H_j+aR_j$.  Therefore
\begin{equation}\label{eq:centereddegree}
 J_j=H_j-\frac{R_j}{r}H_0
\end{equation}
is common-line-twist invariant, possibly rational, by direct substitution.
It can detect variation of the graded bundles beyond their pointwise ranks.
No preservation of Ward admissibility under twisting is assumed.

\subsection{The physical residue in terms of cyclic moments}
Write
\[
 [W_m|_B]=r_m+d_m\epsilon,\qquad
 \epsilon=[M]-1,\qquad \epsilon^2=0.
\]
By \eqref{eq:degree-moments}, $R_j=\sum_m m^j r_m$ and
$H_j=\sum_m m^j d_m$. The convention $m^0=1$ is retained at weight zero.

\begin{Corollary}[Exact physical residue formula]
\label{cor:physical-residue-moments}
For a modification of locally split constant normal type,
\begin{equation}\label{eq:physical-residue-moments}
 \Res_B=R_1+(H_1-R_2+R_1)\epsilon.
\end{equation}
Its rank is $R_1$ and its degree is $H_1-R_2+R_1$.
\end{Corollary}
\begin{proof}
Since $N=M^2=1+2\epsilon$, the signed finite factors in
\eqref{eq:normal-factors} satisfy
\[
 B_m(N)=m-m(m-1)\epsilon\qquad(m\in\Z).
\]
For positive $m$, sum $N^{-a}=1-2a\epsilon$ over $0\leq a<m$.
The negative definition gives the same polynomial for $m<0$, and the
identity is immediate at $m=0$. Multiplying by $r_m+d_m\epsilon$ and
summing in Theorem~\ref{thm:profile-residue} proves the assertion.
\end{proof}

Thus the stable residue retains two specified combinations of the profile
moments. A common degree-$a$ line twist changes its degree by $aR_1$,
as required by the module property of the relative residue.
For the symmetric Coulomb family $R_1=n$, $H_1=n^2$, and $R_2=n^2$,
recovering $n[M]$. For the $SL_2$ family in Theorem~\ref{thm:SL2-family},
$R_1=0$, $R_2=2$, and $H_1=2d$, so its stable degree is $2d-2$.
Section~\ref{sec:mixed-lines} gives a geometric interpretation of this
normal-bundle correction.

\subsection{Complete finite data and integral line twists}
Rank and degree determine a class on $B$, extending
Corollary~\ref{cor:reconstruction} to the full profile.

\begin{Theorem}[Finite profile reconstruction and common line twists]
\label{thm:profile-completeness}
Suppose an actual nonzero profile has distinct occupied weights
$m_1,\ldots,m_s$, positive ranks $r_1,\ldots,r_s$, and degrees
$d_1,\ldots,d_s$:
\[
 \Xi_B(t)=\sum_{i=1}^s(r_i+d_i\epsilon)t^{m_i},\qquad
 r=\sum_i r_i>0.
\]
Then the following statements hold.
\begin{enumerate}[label=(\roman*),leftmargin=*]
\item The moments $R_0=r,R_1,\ldots,R_r$ determine the occupied weights
and their multiplicities. Once these are known, $H_0,\ldots,H_{s-1}$
determine all $d_i$. These finitely many cyclic pairings reconstruct
the full $K^0(B)$-valued profile.
\item At fixed rank polynomial $P(t)=\sum_i r_i t^{m_i}$, the data
\begin{equation}\label{eq:complete-twist-data}
 J_1,\ldots,J_{s-1},\qquad H_0\pmod r
\end{equation}
classify the profile modulo common line twists by $\calO_B(a)$,
$a\in\Z$.
\item If $g=\gcd(r_1,\ldots,r_s)$, the integral degree-vector quotient is
\begin{equation}\label{eq:integral-twist-quotient}
 \Z^s/\Z(r_1,\ldots,r_s)\cong\Z^{s-1}\oplus\Z/g\Z.
\end{equation}
The centred rational moments alone miss the possible finite torsion.
If $g=1$, the congruence datum in \eqref{eq:complete-twist-data} is
redundant.
\end{enumerate}
\end{Theorem}
\begin{proof}
Corollary~\ref{cor:reconstruction} recovers the repeated exponent multiset.
For the distinct occupied weights, the equations
\[
 H_j=\sum_{i=1}^s m_i^j d_i,\qquad 0\leq j<s,
\]
have an invertible Vandermonde matrix, proving (i).
A common degree-$a$ twist sends $d_i$ to $d_i+ar_i$ and $H_0$ to
$H_0+ar$, and leaves every $J_j$ unchanged. Conversely, equal centred
moments, together with their zero-th value zero, give by the same
Vandermonde inversion
\[
 d_i'-d_i=c r_i,\qquad
 c=\frac{H_0'-H_0}{r}\in\mathbb Q.
\]
Equality of the congruence datum forces $c\in\Z$, proving (ii).
Smith normal form of the column $(r_1,\ldots,r_s)^T$ gives
\eqref{eq:integral-twist-quotient}. If $g=1$, B\'ezout's identity shows
that integrality of all $cr_i$ already forces $c$ to be an integer.
In general the kernel of the rational centred-moment map on the integral
quotient is its torsion subgroup, which proves (iii).
\end{proof}

\begin{Example}\label{ex:integral-twist-torsion}
For weights $(0,1)$ and ranks $(2,2)$, the degree vectors $(0,0)$ and
$(1,1)$ both have $J_1=0$. They are not related by an integer common line
twist. Their total degrees are $0$ and $2$ modulo $4$, so
\eqref{eq:complete-twist-data} distinguishes them.
\end{Example}

Completeness concerns actual topological $K$-valued profiles, not arbitrary
moment lists, holomorphic splitting types, extensions or jets.  The twist
quotient is in the vector-bundle category: twisting changes the determinant
by $F^r$, so it may preserve neither a fixed principal $SL_r$ structure nor
Ward line triviality.

\section{Fixed-type families and relative degree pairings}
\label{sec:family}
The following $GL_2$ and determinant-trivial $SL_2$ families have fixed
normal type but varying degree pairings.  Section~\ref{sec:mixed-lines}
relates the latter family's relative degree to mixed-line triviality.

For $d\geq0$, choose the surjection on $B=\CP^1$
\begin{equation}\label{eq:surjection-family}
 \calO_B^2\longrightarrow\calO_B(d),\qquad
 (u,v)\longmapsto\lambda_0^d u+\lambda_1^d v.
\end{equation}
For $d=0$ use the same constant expression.  The two sections have no common
zero, so its kernel $F_d$ is a line bundle of degree $-d$ and hence
$F_d\cong\calO_B(-d)$.  Pull back to $D_0$ and form on $X_0$ the elementary
modification
\begin{equation}\label{eq:elementary-family}
 \mathscr L_-^{(d)}=
 \ker\bigl(\calO_{X_0}^2\longrightarrow i_*p^*\calO_B(d)\bigr),
 \qquad\mathscr L_+=\calO_{X_0}^2,
\end{equation}
where $i:D_0\hookrightarrow X_0$ and the morphism first restricts to the
divisor.  Off $D_0$ the inclusion is an isomorphism.  Elementary
modifications and gluing along Cartier divisors are standard
\cite{Hartshorne1977,BeauvilleLaszlo1995}.

\begin{Theorem}\label{thm:family}
The modification \eqref{eq:elementary-family} is locally free and has normal
exponents $(0,1)$ for every $d$ and at every point of $D_0$.  Its graded
bundles on $B$ are
\[
 W_0|_B=\calO_B(-d),\qquad W_1|_B=\calO_B(d).
\]
Thus $P(t)=1+t$ and $R_j=1$ for every $j\geq1$, independently of $d$,
whereas
\begin{equation}\label{eq:family-pairing}
 \frac12\langle\Phi^{\rel},\Src([W_1|_B])\rangle=d,
 \qquad J_j=d\quad(j\geq1).
\end{equation}
Different values of $d$ therefore give different filtered $K$-profiles and
are distinguished by a common-line-twist-invariant cyclic expression.
\end{Theorem}
\begin{proof}
Locally the surjection in \eqref{eq:surjection-family} splits.  Choose frames
in which its restriction to $D_0$ is projection onto the second factor and
let $z$ define the divisor.  The kernel in \eqref{eq:elementary-family} is
then generated by $e_1$ and $ze_2$.  This proves local freeness and exponents
$(0,1)$.  Reduction modulo the divisor identifies $F_0$ with $p^*F_d$,
while $F_1=\calO_{D_0}^2$.  Hence $W_0=p^*F_d$ and
$W_1=p^*\calO_B(d)$.  Restrict to $B$ and apply
Corollary~\ref{cor:rank-degree}.  One has $H_0=-d+d=0$, and
$H_j=d$ for $j\geq1$, so \eqref{eq:centereddegree} gives $J_j=d$.
\end{proof}

Theorem~\ref{thm:profile-residue} also gives $\Res_B=\calO_B(d)$, so this
degree distinction is already visible to the stable residue.  The profile is
\[
 \Xi_B(t)=(1-d\epsilon)+(1+d\epsilon)t
          =1+t+d\epsilon(t-1),\qquad\epsilon^2=0.
\]
Its nilpotent coefficient is detected by the cyclic degree, but not by
rank moments.  Geometrically, \eqref{eq:surjection-family} maps $B$ with
degree $d$ to the minuscule orbit $\CP^1$ of this fixed $GL_2$ type
\cite[Example~2.1.12 and Lemma~2.1.13, arXiv~v2, p.~27]{Zhu2016}.

\subsection{A determinant-trivial family of fixed semisimple type}
The degree construction is not restricted to determinant charge.  Let
$\mathscr I=\calO_{X_0}(-D_0)$ and let $2D_0$ denote the closed subscheme
defined by $\mathscr I^2$.  Use the same surjection
\eqref{eq:surjection-family} and put
\begin{equation}\label{eq:SL2-family}
 \begin{split}
 \widetilde{\mathscr L}^{(d)}&=
 \ker\bigl(\calO_{X_0}^2\longrightarrow
 p^*\calO_B(d)\otimes\calO_{2D_0}\bigr),\\
 \mathscr L_{-,\mathrm{SL}}^{(d)}&=
 \mathscr I^{-1}\widetilde{\mathscr L}^{(d)}
 \subset\calO_{X_0}(*D_0)^2,
 \qquad\mathscr L_+=\calO_{X_0}^2.
 \end{split}
\end{equation}
Here $\mathscr I^{-1}$ acts as the fractional ideal in the meromorphic
bundle; the gluing off $D_0$ is the resulting inclusion.

\begin{Theorem}\label{thm:SL2-family}
The lattices in \eqref{eq:SL2-family}, with their determinant identification,
define principal $SL_2$ modifications of constant normal type $(-1,1)$.
Their graded bundles on $B$ are
\[
 W_{-1}|_B=\calO_B(-d),\qquad W_1|_B=\calO_B(d).
\]
All their rank polynomials are $t^{-1}+t$.  Nevertheless
\begin{equation}\label{eq:SL2-degree}
 \frac12\langle\Phi^{\rel},C_1^{\wt}\rangle=J_1=2d.
\end{equation}
Moreover, the two sheetwise bundles have the same ordinary class
$[E_+]=[E_-]=2$ in $K^0(X_0)$, while their relative residue is
\begin{equation}\label{eq:SL2-residue}
 \Res_B^{(d)}=[\calO_B(d)]-[\calO_B(2-d)]
             =(2d-2)\epsilon.
\end{equation}
Consequently
\begin{equation}\label{eq:SL2-relative-pairing}
 \frac12\langle\Phi^{\rel},\Src(\Res_B^{(d)})\rangle=2d-2.
\end{equation}
Thus fixed semisimple cocharacter type and ordinary sheetwise $K$-classes do
not determine the supported relative or filtered cyclic degree data.
\end{Theorem}
\begin{proof}
In local frames adapted to the surjection and a normal coordinate $z$,
$\widetilde{\mathscr L}^{(d)}$ is generated by $e_1,z^2e_2$.
After multiplication by $\mathscr I^{-1}$, the generators are
$z^{-1}e_1,ze_2$.  This proves local freeness and exponents $(-1,1)$.
Their wedge is $e_1\wedge e_2$; globally
\[
 \det\widetilde{\mathscr L}^{(d)}=\mathscr I^2,
 \qquad \det\mathscr L_{-,\mathrm{SL}}^{(d)}
       =\mathscr I^{-2}\otimes\mathscr I^2=\calO_{X_0}.
\]
The gluing respects the canonical identification with
$\det\mathscr L_+$, so special-linear frame bundles give the asserted
principal modification.  Formula~\eqref{eq:normal-filtration} gives
$W_{-1}=p^*F_d$ and $W_1=p^*\calO_B(d)$.
Therefore $R_1=-1+1=0$, $H_0=-d+d=0$, and
$H_1=(-1)(-d)+d=2d$.  Equations~\eqref{eq:degree-moments} and
\eqref{eq:centereddegree} prove \eqref{eq:SL2-degree}.

The kernel in \eqref{eq:SL2-family} is the preimage of
$\mathscr I^2p^*\calO_B(d)$ under $\calO_{X_0}^2\to p^*\calO_B(d)$.
Its exact sequence, tensored by $\mathscr I^{-1}=p^*M^2$, gives
\begin{equation}\label{eq:SL2-bundle-sequence}
 0\longrightarrow p^*\calO_B(2-d)\longrightarrow E_-
 \longrightarrow p^*\calO_B(d-2)\longrightarrow0.
\end{equation}
For the relative residue, apply Theorem~\ref{thm:profile-residue} with
$B_{-1}(N)=-N$, $B_1(N)=1$ and $N=M^2$.  This gives
$\Res_B^{(d)}=[\calO_B(d)]-[M^2\otimes\calO_B(-d)]$, proving
\eqref{eq:SL2-residue} in the fixed Koszul convention.
By deformation retraction, $K^0(X_0)\cong K^0(B)$; the exact sequence gives
\[
 [E_-]=[\calO_B(2-d)]+[\calO_B(d-2)]=2=[E_+].
\]
Equation~\eqref{eq:SL2-relative-pairing} follows from
Corollary~\ref{cor:rank-degree}.  In particular the relative computation
cannot be recovered from the difference of these ordinary sheetwise classes.
\end{proof}

For this family, the profile is
$(1-d\epsilon)t^{-1}+(1+d\epsilon)t$.
For $d\geq1$, the next section identifies the value $2d-2$ geometrically;
it also determines exactly which transverse mixed lifts carry a trivial
restricted bundle for every $d\geq0$.
These invariants retain the specified $K$-data beyond ordinary sheetwise
classes, while higher jets remain outside their scope.

\section{Mixed twistor lifts and the residue obstruction}
\label{sec:mixed-lines}
This section tests the geometric information retained by supported cyclic
classes. For pulled-back balanced modifications, the first jet controls
mixed-line splitting and, under generic nondegeneracy, identifies the
residue with a torsion cokernel. The general jet criterion and Euler
example then exhibit holomorphic information beyond this stable class.
Throughout, $B=\CP^1$, $\mathscr I=\calO_{X_0}(-D_0)$, and
\[
 N=N_{D_0/X_0}|_B\cong M^2\cong T_B.
\]
This last identification is specific to the physical source geometry.

\subsection{A first-jet criterion in arbitrary even rank}
Let $k\geq1$ and fix an exact sequence of holomorphic vector bundles
\begin{equation}\label{eq:balanced-quotient}
 0\longrightarrow F\longrightarrow\calO_B^{2k}
 \xrightarrow{\nu}Q\longrightarrow0,
 \qquad \rk F=\rk Q=k.
\end{equation}
Write $e=\deg Q$, so $\deg F=-e$.  Pull back to $X_0$ and set
\begin{equation}\label{eq:balanced-modification}
 E_+=\calO_{X_0}^{2k},\qquad
 E_-=\mathscr I^{-1}
 \ker\bigl(\calO_{X_0}^{2k}\longrightarrow
                  p^*Q\otimes\calO_{2D_0}\bigr).
\end{equation}
The meromorphic inclusion supplies the gluing. The standard flat
connection on $\calO_B^{2k}$ induces the second fundamental form
\begin{equation}\label{eq:second-fundamental-map}
 \beta:T_B\otimes F\longrightarrow Q.
\end{equation}
Differentiate a section of $F$ and project to $Q$; the Leibniz term
lies in $F$, making $\beta$ $\calO_B$-linear.

Let an ordinary section meet $D_0$ transversely at distinct base points
$b,b'$. Its projection to $B$ is an isomorphism. The mixed lift choosing
the minus sheet at $b$ is obtained by using that sheet on
$B\setminus\{b'\}$ and the plus sheet on $B\setminus\{b\}$.
These maps glue because their overlap lies in $U=X_0\setminus D_0$.
Write $E_b$ for its restricted bundle
\cite[Section~2.2, published p.~9, first paragraph after (2.22)]{GarnerPaquette2025}.

\Needspace{7\baselineskip}
\begin{Theorem}[Residue and degeneracy of the first jet]
\label{thm:jet-residue}
For the modification \eqref{eq:balanced-modification}, the following hold.
\begin{enumerate}[label=(\roman*),leftmargin=*]
\item It is a principal $SL_{2k}$ modification of constant normal type
$(-1^k,1^k)$.  Its graded bundles on $B$ are $W_{-1}=F$, $W_1=Q$, and
\begin{equation}\label{eq:balanced-residue}
 \Res_B=[Q]-[N\otimes F].
\end{equation}
\item The bundle on the mixed lift selected at $b$ is
\begin{equation}\label{eq:mixed-lift-bundle}
 E_b\cong\calO_B(b)\otimes
 \ker\bigl(\calO_B^{2k}\xrightarrow{\nu}Q|_{2b}\bigr).
\end{equation}
If $c_b=\dim\ker\beta_b$, then
\begin{equation}\label{eq:mixed-lift-splitting}
 E_b\cong\calO_B(1)^{c_b}\oplus
 \calO_B^{2k-2c_b}\oplus\calO_B(-1)^{c_b}.
\end{equation}
In particular the restricted bundle is holomorphically trivial exactly
when $\beta_b$ is invertible.
\item If $\beta$ is generically invertible, its cokernel $\mathcal T$ is a
torsion sheaf and
\begin{equation}\label{eq:jet-residue-length}
 \Res_B=[\mathcal T]\quad\text{in }K^0(B),\qquad
 \frac12\langle\Phi^{\rel},\Src(\Res_B)\rangle
 =\operatorname{length}(\mathcal T)=2e-2k.
\end{equation}
Here $[\mathcal T]$ is the class supplied by its two-term vector-bundle
resolution.  The support of $\mathcal T$ is precisely the set of base
points producing nontrivial mixed-lift splitting.  Its length counts local
determinantal multiplicities, rather than merely the fiber coranks $c_b$.
\end{enumerate}
\end{Theorem}
\begin{proof}
For (i), adapted local frames give kernel generators
$e_1,\ldots,e_k,z^2e_{k+1},\ldots,z^2e_{2k}$.
Multiplication by $\mathscr I^{-1}$ gives normal type $(-1^k,1^k)$
and cancels the kernel's determinant factor $\mathscr I^{2k}$.
The resulting determinant trivialization respects the gluing.
The normal graded bundles are $F,Q$, and
Theorem~\ref{thm:profile-residue}, with $B_{-1}(N)=-N$ and $B_1(N)=1$,
gives \eqref{eq:balanced-residue}.

For (ii), apply the independently proved general criterion
(Theorem~\ref{thm:general-mixed-criterion}) with
$\widetilde F=(p^*F)|_{2D_0}$ and
$G_\sigma=\sigma^*E_+=\calO_B^{2k}$. Since
$p\circ\sigma=\id_B$, one has $\sigma^*p^*Q=Q$.
Although $\sigma^{-1}(2D_0)=2b+2b'$ globally, the selected minus
lattice uses only $2b$; the plus lattice is unmodified at $b'$.
Thus $T_{\sigma,b}=Q|_{2b}$, proving \eqref{eq:mixed-lift-bundle}.
The canonical flat trivialization is the standard one, and differentiation
of the pulled-back $F$ identifies $\tau_{\sigma,b}$ with $\beta_b$.
Equation~\eqref{eq:general-mixed-splitting} gives
\eqref{eq:mixed-lift-splitting} and its triviality criterion.

For (iii), generic invertibility on the smooth curve gives
\[
 0\longrightarrow T_B\otimes F\xrightarrow{\beta}Q
 \longrightarrow\mathcal T\longrightarrow0
\]
with torsion $\mathcal T$. An isomorphism $N\cong T_B$ identifies its
class with \eqref{eq:balanced-residue}; another differs by a nonzero
constant and leaves both class and degeneracy locus unchanged. Its length is
\[
 \deg Q-\deg(T_B\otimes F)=e-(-e+2k)=2e-2k.
\]
Corollary~\ref{cor:rank-degree} gives the cyclic equality. Local lengths
are the sums of the Smith-factor valuations, hence the orders of
$\det\beta$; their support is the jumping locus identified in (ii).
\end{proof}

Invertibility of this plane-to-quotient derivative is the complex
holomorphic fanning condition \cite[Definition~1.1]{AlvarezPaivaDuran2009}.
The theorem identifies its degeneracy class with the normal residue and
supported cyclic pairing. For $k>1$, a nonconstant Grassmannian map need
not make $\beta$ generically invertible, and its determinant-degeneracy
divisor need not be its ramification divisor.

\subsection{Ramification and the rank-two family}
\begin{Corollary}[Ramification interpretation]
\label{cor:ramification-residue}
Let $k=1$ and let $\nu=(P_0,P_1)$ be a pair of homogeneous degree-$d$
polynomials without a common zero, defining a nonconstant map
$f=[P_0:P_1]:B\to\CP^1$, with $d\geq1$.  Then
$F\cong\calO_B(-d)$ and $Q\cong\calO_B(d)$.  If $R_f$ is the
ramification divisor of $f$, then
\begin{equation}\label{eq:ramification-residue}
 \Res_B=[\calO_B(d)\otimes\calO_{R_f}],\qquad
 \frac12\langle\Phi^{\rel},\Src(\Res_B)\rangle
 =\deg R_f=2d-2.
\end{equation}
The mixed lift selecting its minus-sheet intersection at $b$ is trivial
exactly when $f$ is unramified at $b$.  At a ramification point its
splitting is $\calO_B(1)\oplus\calO_B(-1)$.
\end{Corollary}
\begin{proof}
The second fundamental form is the differential of the corresponding map
to the quotient projective line, expressed as $T_B\otimes F\to Q$.
Locally its coefficient is, up to sign and a unit,
$P_0P_1'-P_1P_0'$.  It is a nonzero morphism
$\calO_B(2-d)\to\calO_B(d)$ whose zero divisor is $R_f$.
Its cokernel is $\calO_B(d)\otimes\calO_{R_f}$, so the theorem gives
the residue and splitting assertions.  Taking degrees gives $2d-2$,
the Riemann--Hurwitz formula for a map of projective lines
\cite[Section~53.12, Tag~0C1B]{StacksRiemannHurwitz}.
\end{proof}

The sheaf in \eqref{eq:ramification-residue} retains the coefficient
$\calO_B(d)|_{R_f}$; its topological $K$-class retains only the total
length. Every ramification point gives the same mixed splitting, whereas
its multiplicity belongs to the family's jet-degeneracy divisor.

\begin{Corollary}[Triviality on all transverse mixed lifts]
\label{cor:mixed-line-triviality}
For the data of Theorem~\ref{thm:jet-residue}, conditions
\textup{(i)} and \textup{(ii)} below are equivalent without a
generic-invertibility assumption. If $\beta$ is generically invertible,
put $\mathcal T=\operatorname{coker}\beta$; then all five conditions
are equivalent:
\begin{enumerate}[label=(\roman*),leftmargin=*]
\item every mixed lift of every transverse ordinary section has a trivial
restricted bundle;
\item $\beta$ is invertible at every point of $B$;
\item $\mathcal T=0$;
\item $\frac12\langle\Phi^{\rel},\Src(\Res_B)\rangle=0$;
\item $\deg Q=k$.
\end{enumerate}
For the nonconstant rank-two family these conditions hold exactly when
$d=1$.
\end{Corollary}
\begin{proof}
For an ordinary section with
$u=a_0+a_1\zeta$ and $v=b_0+b_1\zeta$, its divisor equation is
\[
 q|_L=b_0+(b_1-a_0)\zeta-a_1\zeta^2.
\]
These sections range over all quadratic sections of $\calO_B(2)$.
Every base point therefore occurs as a simple intersection of some
transverse section. Theorem~\ref{thm:jet-residue}(ii) gives
\textup{(i)}$\Leftrightarrow$\textup{(ii)}.
Under generic invertibility, the remaining equivalences follow from
the torsion-cokernel description and its nonnegative length.
In rank two, \eqref{eq:ramification-residue} reduces the last condition
to $d=1$.
\end{proof}

For the maps $f_d=[\lambda_0^d:\lambda_1^d]$ in
Theorem~\ref{thm:SL2-family}, $d=1$ has no ramification.  For $d\geq2$,
the ramification divisor is
\[
 R_{f_d}=(d-1)[0]+(d-1)[\infty].
\]
Thus a selected mixed lift is trivial exactly when its minus-sheet
intersection avoids $0$ and $\infty$; both mixed lifts of a transverse
section are trivial exactly when both intersections avoid these points.
The nonzero residue degree $2d-2$ counts an additional jumping-line
obstruction.  For the constant case $d=0$, $\beta=0$, every mixed lift
has splitting $\calO_B(1)\oplus\calO_B(-1)$, and the effective
ramification interpretation does not apply.  Its residue remains
$-2\epsilon$, as in \eqref{eq:SL2-residue}.

\subsection{Comparison with ordinary sheetwise restrictions}
\begin{Proposition}\label{prop:ordinary-section-splittings}
For the rank-two family \eqref{eq:SL2-family}, the restriction of $E_-$ to
every ordinary twistor section $L\cong B$, including sections contained
in $D_0$, is
\[
 E_-^{(d)}|_L\cong
 \begin{cases}
 \calO_B(2)\oplus\calO_B(-2),&d=0,\\
 \calO_B(1)\oplus\calO_B(-1),&d=1,\\
 \calO_B^{\oplus2},&d=2.
 \end{cases}
\]
\end{Proposition}
\begin{proof}
The quotient in \eqref{eq:SL2-bundle-sequence} is locally free, so its
restriction to any section is exact:
\[
 0\longrightarrow\calO_B(2-d)\longrightarrow E_-^{(d)}|_L
 \longrightarrow\calO_B(d-2)\longrightarrow0.
\]
For $d=0,1,2$, its extension group is
$H^1(B,\calO_B(4-2d))=0$.  Hence the sequence splits.
\end{proof}

Thus $d=2$ has trivial ordinary sheetwise restrictions but mixed jumping
at ramification points, whereas $d=1$ has nontrivial ordinary minus-sheet
restrictions and trivial transverse mixed restrictions.

\subsection{General balanced modifications on the first neighbourhood}
\label{subsec:general-balanced-data}
We now allow dependence on the source coordinate and first normal jet,
with arbitrary plus bundle.

Write $i_2:2D_0\hookrightarrow X_0$ for the closed subspace defined by
$\mathscr I^2$.  Let $E_+$ be a holomorphic bundle of rank $2k$ on $X_0$,
and let $\widetilde F\subset E_+|_{2D_0}$ be a rank-$k$ locally free
subbundle whose quotient is locally free. Thus the inclusion splits
locally; a global holomorphic splitting is not required. Set
\begin{equation}\label{eq:general-balanced-data}
 \widetilde Q=E_+|_{2D_0}/\widetilde F,\qquad
 F_D=\widetilde F|_{D_0},\qquad Q_D=\widetilde Q|_{D_0},
 \qquad
 E_-(\widetilde F)=\mathscr I^{-1}
       \ker(E_+\longrightarrow i_{2*}\widetilde Q).
\end{equation}
The resulting meromorphic inclusion identifies the two bundles off $D_0$.
Denote the glued bundle on the ordered double by $\mathcal V_{\widetilde F}$.

\begin{Lemma}[First-neighbourhood description]
\label{lem:balanced-first-neighbourhood}
For fixed $E_+$, construction \eqref{eq:general-balanced-data} gives a
bijection between these subbundles on $2D_0$ and embedded holomorphic
lattices of locally split constant normal type $(-1^k,1^k)$ relative to
$E_+$.  Their graded bundles are $W_{-1}=F_D$ and $W_1=Q_D$.  Their
determinant lattice equals $\det E_+$; thus a chosen trivialization of
$\det E_+$ makes the modification principal $SL_{2k}$.  On the compact
source slice $B=B\times\{0\}\subset D_0$ its stable residue is
\begin{equation}\label{eq:general-balanced-residue}
 \Res_B(\mathcal V_{\widetilde F})
   =[Q_D|_B]-[N\otimes F_D|_B],\qquad
 N=N_{D_0/X_0}|_B\cong\calO_B(2).
\end{equation}
\end{Lemma}
\begin{proof}
Put $\mathscr L_+=E_+$ and regard $\mathscr L_-$ as a lattice in
$E_+(*D_0)$.  The stated normal type gives
\[
 \mathscr I\mathscr L_+\subset\mathscr L_-
       \subset\mathscr I^{-1}\mathscr L_+.
\]
These inclusions can be checked after completion, as in the proof of
Theorem~\ref{thm:profile-residue}.  The quotient
\[
 \widetilde F=\mathscr I\mathscr L_-/\mathscr I^2\mathscr L_+
       \subset\mathscr L_+/\mathscr I^2\mathscr L_+
\]
is a locally split subbundle of rank $k$: modulo $\mathscr I^2$, the
normal form is the span of the first $k$ columns of its invertible matrix.
Only its first two normal coefficients are involved.  Conversely, a local
frame of such a subbundle on $2D_0$ extends locally to an invertible frame
of $E_+$.  If $D_0=(z=0)$, the kernel in
\eqref{eq:general-balanced-data} then has generators
$e_1,\ldots,e_k,z^2e_{k+1},\ldots,z^2e_{2k}$.  Multiplication by
$\mathscr I^{-1}$ gives $z^{-1}e_i$ and $ze_j$.  This proves both inverse
constructions and the determinant assertion.  Reducing the normal
filtration gives $F_D,Q_D$, and Theorem~\ref{thm:profile-residue} gives
\eqref{eq:general-balanced-residue}.
\end{proof}

This bijection retains the embedded lattice, including its first normal
jet and framing data; the normal type alone does not determine it.

\subsection{A general mixed-line jet criterion}
\label{subsec:general-mixed-criterion}
Let $\sigma:B\to X_0$ be an ordinary section meeting $D_0$ transversely
at distinct points with base coordinates $b,b'$.  Select the mixed lift
$\widetilde L_{\sigma,b}$ using the minus sheet at $b$, and put
$o=\sigma(b)$ and $G_\sigma=\sigma^*E_+$.  The full inverse image of
$2D_0$ is $2b+2b'$.  In the following construction we pull back
$\widetilde F$ only along the selected map $2b\to2D_0$, obtaining
\[
 \widetilde F_{\sigma,b}\subset G_\sigma|_{2b},\qquad
 T_{\sigma,b}=G_\sigma|_{2b}/\widetilde F_{\sigma,b}.
\]
Here $T_{\sigma,b}$ is a sheaf of length $2k$ supported at $b$.
Define the jet-evaluation map
\begin{equation}\label{eq:general-jet-evaluation}
 \mathcal J_{\sigma,b}:H^0(B,G_\sigma)
          \longrightarrow H^0(B,T_{\sigma,b}).
\end{equation}

\begin{Theorem}[General first-jet criterion]
\label{thm:general-mixed-criterion}
For the modification \eqref{eq:general-balanced-data}, set
$K_{\sigma,b}=\ker(G_\sigma\to T_{\sigma,b})$.  Then
\begin{equation}\label{eq:general-mixed-restriction}
 \mathcal V_{\widetilde F}|_{\widetilde L_{\sigma,b}}
       \cong\calO_B(b)\otimes K_{\sigma,b},\qquad
 \deg\bigl(\mathcal V_{\widetilde F}|_{\widetilde L_{\sigma,b}}\bigr)
       =\deg G_\sigma.
\end{equation}
The restricted bundle is trivial if and only if $\deg G_\sigma=0$ and
$\mathcal J_{\sigma,b}$ is injective.  In that case the jet-evaluation
map is bijective and $H^1(B,G_\sigma)=0$.

If $G_\sigma$ is trivial, evaluation of its global sections supplies a
canonical flat trivialization.  In this trivialization the first
variation of $\widetilde F_{\sigma,b}$ defines
\begin{equation}\label{eq:general-first-variation}
 \tau_{\sigma,b}:T_bB\otimes F_{D,o}\longrightarrow Q_{D,o}.
\end{equation}
Writing $c_{\sigma,b}=\dim\ker\tau_{\sigma,b}$, the exact splitting is
\begin{equation}\label{eq:general-mixed-splitting}
 \mathcal V_{\widetilde F}|_{\widetilde L_{\sigma,b}}
 \cong\calO_B(1)^{c_{\sigma,b}}\oplus
       \calO_B^{2k-2c_{\sigma,b}}\oplus
       \calO_B(-1)^{c_{\sigma,b}}.
\end{equation}
In particular it is trivial exactly when $\tau_{\sigma,b}$ is invertible.
\end{Theorem}
\begin{proof}
Away from $b$ the lift uses the plus lattice $G_\sigma$.  At $b$, a local
source equation pulls back to a local parameter times a unit, so the
minus lattice is $(\zeta-b)^{-1}K_{\sigma,b}$.  Units do not change the
lattice.  This proves \eqref{eq:general-mixed-restriction}; the other
intersection is on the plus sheet and produces no second modification.
The kernel has degree $\deg G_\sigma-2k$, and tensoring its rank-$2k$
bundle by $\calO_B(b)$ restores $\deg G_\sigma$.

By the Birkhoff--Grothendieck theorem \cite{Grothendieck1957}, a
degree-zero bundle $E$ on $B$ is trivial precisely when
$H^0(B,E(-1))=0$.  Since $\calO_B(b)\otimes\calO_B(-1)$ is trivial,
this cohomology group is $H^0(B,K_{\sigma,b})$, the kernel of
\eqref{eq:general-jet-evaluation}.  If $\deg G_\sigma=0$ and that map
is injective, then $h^0(G_\sigma)\leq2k$, while Riemann--Roch gives
$h^0(G_\sigma)=2k+h^1(G_\sigma)$.  Bijectivity and the stated vanishing
follow.

For trivial $G_\sigma$, the flat trivialization identifies its fibres
with $\mathbb V=H^0(B,G_\sigma)$.  For $v\in T_bB$ and $f\in F_{D,o}$,
choose a local section $g$ with $g(b)=f$ whose first jet lies in
$\widetilde F_{\sigma,b}$.  Define
$\tau_{\sigma,b}(v\otimes f)=\nabla_vg\bmod F_{D,o}$.  A change of
$g$ with the same value changes its derivative by a vector of $F_{D,o}$,
so this is well defined.  The inclusions
\[
 \calO_B(-b)\otimes\mathbb V\subset
 \calO_B(b)\otimes K_{\sigma,b}\subset
 \calO_B(b)\otimes\mathbb V
\]
bound all splitting degrees between $-1$ and $1$: the right inclusion
bounds the largest degree, and projection of the left inclusion to a
smallest-degree summand gives the lower bound. The number of positive
summands is $h^0(K_{\sigma,b})$.  The value part of jet evaluation is
surjective with kernel $F_{D,o}$, and its derivative on that kernel is
$-\tau_{\sigma,b}$ after choosing a nonzero tangent vector.  Thus
$h^0(K_{\sigma,b})=c_{\sigma,b}$.  Degree zero gives the same number of
negative summands, proving \eqref{eq:general-mixed-splitting}.
\end{proof}

For nontrivial $G_\sigma$, the cohomological test in the first part of the
theorem remains the criterion; the three-term splitting in
\eqref{eq:general-mixed-splitting} is asserted only when $G_\sigma$ is
trivial.  For example, the usual charge-one Coulomb embedding has
\[
 E_+=M^{\oplus k}\oplus(M^{-1})^{\oplus k},\qquad
 \widetilde F=(M^{-1})^{\oplus k}|_{2D_0}.
\]
Here $G_\sigma=\calO_B(1)^k\oplus\calO_B(-1)^k$ and
$\mathcal J_{\sigma,b}$ is first-jet evaluation of $k$ linear forms,
which is bijective.  Hence every transverse mixed restriction is trivial.
These are the cocharacter data of \cite[Sections~2--3]{GarnerPaquette2025},
included to illustrate why the hypothesis on $G_\sigma$ matters.

\subsection{The tangent direction and the normal jet}
\label{subsec:total-normal-jet}
Now suppose $E_+=\calO_{X_0}^{2k}$. The subbundle $\widetilde F$
defines a map $2D_0\to\operatorname{Gr}(k,2k)$.  Its differential along
the reduced divisor is the intrinsic map
\begin{equation}\label{eq:total-normal-jet}
 \Theta_{\widetilde F}:T_{X_0}|_{D_0}
       \longrightarrow\operatorname{Hom}(F_D,Q_D).
\end{equation}
For $v\in T_oX_0$ and $f\in F_{D,o}$, differentiate a local lift of a
section of $\widetilde F$ with value $f$ and project to $Q_{D,o}$.
Changing an ambient lift by $\mathscr I^2$ has zero derivative on $D_0$;
changing the section while fixing its value changes the derivative only
within $F_{D,o}$.  Thus \eqref{eq:total-normal-jet} is defined on all of
$T_{X_0}|_{D_0}$.  Its restriction to $T_{D_0}$ is the second fundamental
form of $F_D\subset\calO_{D_0}^{2k}$, and
\begin{equation}\label{eq:general-line-jet}
 \tau_{\sigma,b}(v\otimes f)
    =\Theta_{\widetilde F}(d\sigma_b(v))(f).
\end{equation}

On the chart $\lambda_0\ne0$, use
$\zeta=\lambda_1/\lambda_0$, $s=u=\mu_0/\lambda_0$, and
$q=v-\zeta u$, with $v=\mu_1/\lambda_0$.  Write
$\Theta_\zeta=\Theta_{\widetilde F}(\partial_\zeta)$,
$\Theta_s=\Theta_{\widetilde F}(\partial_s)$ and
$\Theta_q=\Theta_{\widetilde F}(\partial_q)$.
For $u=a_0+a_1\zeta$, $v=b_0+b_1\zeta$ and a selected finite root $b$,
\eqref{eq:general-line-jet} is
\begin{equation}\label{eq:general-line-jet-coordinates}
 \tau_{\sigma,b}(\partial_\zeta\otimes-)
   =\Theta_\zeta+a_1\Theta_s+\Delta\Theta_q,
 \qquad \Delta=q_\sigma'(b)=b_1-a_0-2a_1b\ne0.
\end{equation}
This intrinsic derivative includes the section's normal velocity.
Its three components are coordinate-dependent; no global tangent
splitting or retraction $2D_0\to D_0$ is assumed. For
$\widetilde F=(p^*F)|_{2D_0}$, the last two components vanish, giving
Theorem~\ref{thm:jet-residue}.

\begin{Corollary}[A necessary residue condition for all mixed lifts]
\label{cor:general-mixed-residue}
Assume $E_+=\calO_{X_0}^{2k}$ and that every transverse mixed lift has
trivial restricted bundle.  Let
$\rho_N:T_{X_0}|_{D_0}\to N_{D_0/X_0}$ be normal projection.  There is
an integer $0\leq m\leq k$ and a nowhere-vanishing holomorphic section
\[
 c\in H^0\!\left(D_0,
 (\det F_D)^{-1}\otimes\det Q_D\otimes
 p^*K_B^{\otimes(k-m)}\otimes N_{D_0/X_0}^{\otimes(-m)}\right)
\]
such that, as a homogeneous polynomial on each tangent space,
\begin{equation}\label{eq:general-determinant-factorization}
 \det\Theta_{\widetilde F}(v)
       =c\,(dp(v))^{k-m}(\rho_N(v))^m.
\end{equation}
Consequently $\deg(Q_D|_B)=k$ and
\begin{equation}\label{eq:general-necessary-residue}
 \Res_B(\mathcal V_{\widetilde F})=0\quad\hbox{in }K^0(B),
 \qquad
 \frac12\langle\Phi^{\rel},
       \Src(\Res_B(\mathcal V_{\widetilde F}))\rangle=0.
\end{equation}
The vanishing of this residue alone is not a sufficient line-triviality
criterion.
\end{Corollary}
\begin{proof}
At a point $o\in D_0$, the tangent directions of transverse ordinary
sections selected at $o$ are exactly those with $dp(v)\ne0$ and
$\rho_N(v)\ne0$.  Indeed, in the above coordinates, after normalizing
$dp(v)=1$, every pair $(a_1,\Delta)\in\C\times\C^*$ is realized by
an ordinary section through $o$.  At $\zeta=b$, with $s(o)=s_0$, take
\[
 u=s_0+a_1(\zeta-b),\qquad
 v=bs_0+(s_0+ba_1+\Delta)(\zeta-b).
\]
Then $q_\sigma=\Delta(\zeta-b)-a_1(\zeta-b)^2$.  Its selected zero
is simple, and its other projective zero is distinct, possibly at
infinity.  The other base chart covers a selected point at infinity.

Theorem~\ref{thm:general-mixed-criterion} therefore says that
$\det\Theta_{\widetilde F}$ has no zero in these allowed directions.
The projective zero set of this homogeneous degree-$k$ polynomial is
supported on the two distinct hyperplanes $dp(v)=0$ and $\rho_N(v)=0$.
Unique factorization over $\C$ gives
\eqref{eq:general-determinant-factorization} at each point.  Since the
polynomial never vanishes identically, its two exponents are locally
constant, hence constant on the connected divisor $D_0$.  Its nonzero
coefficients glue to the displayed nowhere-vanishing section $c$.

Put $e=\deg(Q_D|_B)$; then $\deg(F_D|_B)=-e$.  On $B$ the line bundle
containing $c$ has degree
\[
 2e-2(k-m)-2m=2e-2k,
\]
using $N\cong T_B\cong\calO_B(2)$.  Its nowhere-vanishing section
forces $e=k$.  The residue in \eqref{eq:general-balanced-residue} has
rank zero and degree $2e-2k$, so it vanishes in $K^0(B)$; its supported
pairing vanishes as well.

For the failure of sufficiency, take the direct sum of the pulled-back
rank-two families with $d=0$ and $d=2$.  Then $k=2$,
$F_D|_B=\calO_B\oplus\calO_B(-2)$ and
$Q_D|_B=\calO_B\oplus\calO_B(2)$, so the residue is zero.
The constant block has zero first variation and contributes
$\calO_B(1)\oplus\calO_B(-1)$ to every mixed restriction.  None of
these restrictions is trivial.
\end{proof}

There is also a conditional torsion interpretation on a fixed source
slice.  Still assuming $E_+=\calO_{X_0}^{2k}$, let
$\iota:B=B\times\{0\}\hookrightarrow D_0$ and restrict
$\Theta_{\widetilde F}$ to its tangent direction.  This gives a global
map
\begin{equation}\label{eq:general-source-slice-jet}
 \beta_B:T_B\otimes F_D|_B\longrightarrow Q_D|_B,\qquad
 \beta_B(v\otimes f)=\Theta_{\widetilde F}(d\iota(v))(f).
\end{equation}
If $\beta_B$ is generically invertible, its torsion cokernel has the
two-term resolution
\[
 0\longrightarrow T_B\otimes F_D|_B
   \xrightarrow{\beta_B}Q_D|_B
   \longrightarrow\operatorname{coker}\beta_B\longrightarrow0.
\]
Consequently \eqref{eq:general-balanced-residue} and $N\cong T_B$ give
\begin{equation}\label{eq:general-source-slice-length}
 \Res_B=[\operatorname{coker}\beta_B],\qquad
 \frac12\langle\Phi^{\rel},\Src(\Res_B)\rangle
   =\operatorname{length}(\operatorname{coker}\beta_B)
   =2\deg(Q_D|_B)-2k.
\end{equation}
This measures degeneracy on the chosen source slice. The additional
terms in \eqref{eq:general-line-jet-coordinates} generally separate
this locus from mixed-line jumping; in rank two, the slice map's
ramification divisor therefore does not suffice. For pulled-back data
those terms vanish, recovering Theorem~\ref{thm:jet-residue}.

\subsection{A rank-four example with a nonzero normal jet}
\begin{Example}\label{ex:rank-four-normal-jet}
Write $Z=[\lambda:\mu]\in X_0$, where
$\lambda=(\lambda_0,\lambda_1)^T\ne0$ and
$\mu=(\mu_0,\mu_1)^T$.  The planes
\begin{equation}\label{eq:rank-four-general-planes}
 \mathscr P_Z=\operatorname{span}
       \{(\lambda;\mu),(0;\lambda)\}\subset\C^4
\end{equation}
define a holomorphic rank-two subbundle $\mathscr P\subset\calO_{X_0}^4$:
the columns are independent and their span is unchanged by the common
projective rescaling.  Take $E_+=\calO_{X_0}^4$ and
$\widetilde F=\mathscr P|_{2D_0}$ in
\eqref{eq:general-balanced-data}.

In the coordinates $(\zeta,s,q)$ an equivalent frame is
\[
 f_1=(1,\zeta,0,q)^T,\qquad f_2=(0,0,1,\zeta)^T.
\]
On $D_0$ these are two Euler subbundles, so
$F_D=p^*\calO_B(-1)^{\oplus2}$ and
$Q_D=p^*\calO_B(1)^{\oplus2}$.  In the quotient frames supplied by
$[\partial_\zeta f_1]$ and $[\partial_\zeta f_2]$ at $q=0$,
\[
 \Theta_\zeta=I_2,\qquad \Theta_s=0,\qquad
 \Theta_q=\begin{pmatrix}0&0\\1&0\end{pmatrix}.
\]
Hence $\det(\Theta_\zeta+a_1\Theta_s+\Delta\Theta_q)=1$ for every
$(a_1,\Delta)$.  Every transverse mixed restriction is trivial by
Theorem~\ref{thm:general-mixed-criterion}, although the normal jet is
nonzero.  Its residue is zero, consistently with
Corollary~\ref{cor:general-mixed-residue}.

The global calculation agrees with the local one.  On the section
$\mu=x\lambda$, where $x$ is a constant $2\times2$ matrix,
\[
 \mathscr P|_{L_x}=g_x(F\oplus F),\qquad
 g_x=\begin{pmatrix}I_2&0\\x&I_2\end{pmatrix},
 \qquad F=\calO_B(-1)\subset\calO_B^2.
\]
Thus the induced first jets on each line are related by a constant change
of frame to two degree-one blocks.

This modification is not isomorphic over the identity of the ordered
double to one whose first-neighbourhood subbundle is pulled back from
$B$.  Every global holomorphic function on $X_0$ is constant: its
positive fibre-degree Taylor coefficients lie in
$H^0(B,\operatorname{Sym}^j(M^{-1}\oplus M^{-1}))=0$ for $j>0$.
A plus-sheet automorphism of $\calO_{X_0}^4$ is therefore a constant
matrix.  Such a change of frame preserves the property
$\Theta_{\widetilde F}|_{T_{X_0/B}}=0$, which holds for pulled-back
data and fails in this example.  The general criterion consequently
extends beyond the pulled-back family even when all mixed restrictions
are trivial.
\end{Example}

\subsection{The Euler block and its explicit complex ASD field}
\label{subsec:euler-field}
For $d=1$, Corollary~\ref{cor:mixed-line-triviality} permits a local
Ward reconstruction. Write
\[
 x=\begin{pmatrix}a_0&a_1\\ b_0&b_1\end{pmatrix},\qquad
 u=a_0+a_1\zeta,\qquad v=b_0+b_1\zeta,
 \qquad q_x(\zeta)=b_0+(b_1-a_0)\zeta-a_1\zeta^2.
\]
On the cover choosing a minus-sheet root of the homogeneous quadratic
$q_x$, take the chart with finite selected root $r$ and put
\begin{equation}\label{eq:euler-root-domain}
 q_x(r)=0,\qquad
 \Delta_x=q_x'(r)=b_1-a_0-2a_1r\ne0,\qquad
 \varrho=-\tfrac12\Delta_x.
\end{equation}
The root $r$ is locally holomorphic; the other root may be infinite. Set
\begin{equation}\label{eq:euler-spacetime-coordinates}
 t=\frac{a_0+b_1}{2},\qquad y^1=\frac{a_0-b_1}{2},\qquad
 y^2=a_1,\qquad y^3=b_0.
\end{equation}
Then $\varrho^2=(y^1)^2+y^2y^3$, and
$\det(dx)=dt^2-(dy^1)^2-dy^2dy^3$ fixes the conformal metric.

\begin{Proposition}[Explicit Euler field]\label{prop:euler-field}
Let
\[
 \mathsf H=\begin{pmatrix}1&0\\0&-1\end{pmatrix},\qquad
 \mathsf N=\begin{pmatrix}0&0\\1&0\end{pmatrix},\qquad
 \mathsf a=\frac{db_1-r\,da_1}{2\varrho}
          =\frac{dt-dy^1-r\,dy^2}{2\varrho}.
\]
The mixed-line Ward transform of the $d=1$ modification
\eqref{eq:SL2-family} is represented, on the domain
\eqref{eq:euler-root-domain}, by the holomorphic $\mathfrak{sl}_2(\C)$
connection
\begin{equation}\label{eq:euler-connection}
 \nabla=d+A,\qquad A=\mathsf H\mathsf a+\mathsf N\,dr.
\end{equation}
It obeys the complex ASD equations in the convention
\begin{equation}\label{eq:euler-asd}
 F_{a_0b_0}=F_{a_1b_1}=0,\qquad
 F_{a_1b_0}+F_{a_0b_1}=0.
\end{equation}
Its curvature is
\begin{align}
 F_A&=\mathsf H\mathsf f+2\mathsf N\,dr\wedge\mathsf a,
       \label{eq:euler-curvature}\\
 \mathsf f=d\mathsf a
   &=\frac{dt\wedge d(\varrho^2)+y^1dy^2\wedge dy^3
       +y^2dy^3\wedge dy^1+y^3dy^1\wedge dy^2}{4\varrho^3}.
       \label{eq:euler-coulomb-curvature}
\end{align}
Here $\mathsf a$ is the charge-one Coulomb potential in the convention
\eqref{eq:symmetric-coulomb}.  The connection preserves a line and has
no parallel complementary line on any nonempty open subset of its domain.
In particular, it is not gauge equivalent to the diagonal connection
$d+\mathsf H\mathsf a$.
\end{Proposition}

\begin{proof}
The quotient $(1,\zeta)$ in the Euler sequence has kernel spanned by
$(-\zeta,1)^T$.  The constant determinant-one ambient change of frame
$\left(\begin{smallmatrix}0&1\\-1&0\end{smallmatrix}\right)$ sends this
kernel to the span of $f(\zeta)=(1,\zeta)^T$.  In that frame a global
frame of the mixed restriction is
\begin{equation}\label{eq:euler-splitting-frame}
 \mathcal H(\zeta;r)
   =\bigl[f(r)\ \big|\ (\zeta-r)^{-1}f(r)+f'(r)\bigr]
   =\begin{pmatrix}
      1&(\zeta-r)^{-1}\\
      r&r(\zeta-r)^{-1}+1
     \end{pmatrix},\qquad \det\mathcal H=1.
\end{equation}
The second column $(\zeta-r)^{-1}f(\zeta)$ lies in the pole direction;
the first equals $f(\zeta)-(\zeta-r)f'(r)$ and lies in the same lattice.
Both are regular at infinity and their determinant is one, proving the
frame assertion including at the selected minus-sheet point.

The incidence vector fields
\[
 V_0=\partial_{a_1}-\zeta\partial_{a_0},\qquad
 V_1=\partial_{b_1}-\zeta\partial_{b_0}
\]
annihilate $u,v$ and commute.  Implicit differentiation of $q_x(r)=0$
gives
\[
 V_0r=-\frac{r(\zeta-r)}{\Delta_x},\qquad
 V_1r=\frac{\zeta-r}{\Delta_x},\qquad
 \mathcal H^{-1}\partial_r\mathcal H
       =-\frac{\mathsf H}{\zeta-r}+\mathsf N.
\]
Consequently the Ward operators are
\begin{align*}
 \mathcal H^{-1}V_0\mathcal H
   &=V_0-\frac r{\Delta_x}
          \bigl(-\mathsf H+(\zeta-r)\mathsf N\bigr),\\
 \mathcal H^{-1}V_1\mathcal H
   &=V_1+\frac1{\Delta_x}
          \bigl(-\mathsf H+(\zeta-r)\mathsf N\bigr).
\end{align*}
Their coefficients are linear in $\zeta$ and their commutator vanishes
by conjugacy. The local Ward construction \cite{Ward1977,WardWells1990}
therefore gives \eqref{eq:euler-asd}, with connection
\[
 A=-\frac{\mathsf H(db_1-r\,da_1)
      +\mathsf N(db_0-r\,da_0+r\,db_1-r^2da_1)}{\Delta_x}.
\]
Since the second one-form in the numerator is $-\Delta_x\,dr$,
this proves \eqref{eq:euler-connection}.  The identities
$[\mathsf H,\mathsf N]=-2\mathsf N$ and $\mathsf N^2=0$ give
\eqref{eq:euler-curvature}; differentiating $\mathsf a$ gives
\eqref{eq:euler-coulomb-curvature}.

The charge-one bundle \eqref{eq:symmetric-coulomb} restricts to
$\calO_B(1)(-r)$ on this mixed line.  Its nowhere-zero section, represented
in the plus-sheet frame by $\zeta-r$, has Ward logarithmic derivatives
$r/\Delta_x$ and $-1/\Delta_x$.  These give precisely $\mathsf a$.
This also identifies the normalization of the diagonal Coulomb term
\cite[Section~3.2]{GarnerPaquette2025}.

For the nonsplitting assertion, $(t,\varrho,r,w)$, where $w=y^2$, are
regular coordinates on the finite-root domain: the inverse formulas are
\begin{equation}\label{eq:euler-adapted-coordinates}
 y^1=\varrho-wr,\qquad y^2=w,\qquad
 y^3=2\varrho r-wr^2,\qquad
 \mathsf a=\frac{dt-d\varrho+w\,dr}{2\varrho}.
\end{equation}
In these coordinates
\[
 F_A(\partial_t,\partial_\varrho)
     =\frac{\mathsf H}{2\varrho^2},\qquad
 F_A(\partial_r,\partial_t)=\frac{\mathsf N}{\varrho},
\]
and therefore
\begin{equation}\label{eq:euler-curvature-commutator}
 \bigl[F_A(\partial_t,\partial_\varrho),
       F_A(\partial_r,\partial_t)\bigr]
       =-\frac{\mathsf N}{\varrho^3}\ne0.
\end{equation}
The line spanned by the second basis vector is parallel.  A parallel
complement would make all curvature endomorphisms diagonal in a common
frame, contradicting \eqref{eq:euler-curvature-commutator}.
The same commutator transforms by conjugacy under any complex gauge
transformation, and excludes gauge equivalence to every embedded abelian
connection, including the diagonal Coulomb part.
\end{proof}

\begin{Corollary}[The second root chart]\label{cor:euler-root-charts}
The Euler connection extends to the entire chosen-root cover of the
complex spacetime region where $\varrho\ne0$.  On the chart where the
selected root has coordinate $s=1/r$, it is represented by
\begin{equation}\label{eq:euler-infinity-connection}
 A_\infty=\mathsf H\mathsf a_\infty-\mathsf N\,ds,
 \qquad
 \mathsf a_\infty=\frac{da_0-s\,db_0}{2\varrho}.
\end{equation}
On the overlap the determinant-one transition is
\begin{equation}\label{eq:euler-gauge-transition}
 g_{0\infty}=\diag(r^{-1},r),\qquad
 A_\infty=g_{0\infty}^{-1}A_0g_{0\infty}
                  +g_{0\infty}^{-1}dg_{0\infty},
 \quad A_0=A.
\end{equation}
The diagonal Coulomb connections $\mathsf H\mathsf a$ and
$\mathsf H\mathsf a_\infty$ glue by the same transition.
\end{Corollary}
\begin{proof}
With $\eta=1/\zeta$, a determinant-one mixed-line frame on the second
chart is
\[
 \mathcal H_\infty(\eta;s)=
 \begin{pmatrix}
 s&-\eta/(\eta-s)\\1&-1/(\eta-s)
 \end{pmatrix}.
\]
Substitution gives
$\mathcal H_\infty(1/\zeta;1/r)=\mathcal H(\zeta;r)g_{0\infty}$,
so the Ward connections have the asserted gauge transition.
Since $g_{0\infty}^{-1}\mathsf Ng_{0\infty}=r^{-2}\mathsf N$,
$ds=-r^{-2}dr$ and
$\mathsf a_\infty=\mathsf a-dr/r$, it gives
\eqref{eq:euler-infinity-connection}.
More explicitly, $(t,\varrho,s,v)$, with $v=y^3$, are coordinates with
\[
 y^1=vs-\varrho,\qquad y^2=2\varrho s-vs^2,
 \qquad
 \mathsf a_\infty=\frac{dt-d\varrho+v\,ds}{2\varrho}.
\]
Thus $A_\infty$ is holomorphic at $s=0$. The two charts cover the
selected projective root, and the ASD equations extend across $s=0$
by holomorphicity. Setting the nilpotent term to zero proves the final
assertion.
\end{proof}

Let $\mathcal C_n$ denote the charge-$n$ line bundle
\eqref{eq:symmetric-coulomb} on the double.  The Euler modification has
an extension
\[
 0\longrightarrow\mathcal C_{-1}\longrightarrow\mathcal V
  \longrightarrow\mathcal C_1\longrightarrow0.
\]
Its plus restriction is the pulled-back Euler sequence, nonsplit already
on the zero section. Proposition~\ref{prop:euler-field} proves that the
local connection is likewise \emph{reducible and indecomposable}.
Although $dr$ is exact, the gauge $g=I_2-r\mathsf N$, in the convention
$A^g=g^{-1}Ag+g^{-1}dg$, replaces $A$ by
$\mathsf a(\mathsf H+2r\mathsf N)$ and does not diagonalize it.
Extension constructions are classical in the Atiyah--Ward ans\"atze
\cite{AtiyahWard1977,WardWells1990}; the calculation here identifies
the field of this Euler modification on the source double.

\begin{Remark}[Reality and scope]\label{rem:euler-reality}
The field is a complex holomorphic ASD connection on the chosen-root
cover off $\varrho=0$, with the duality convention fixed in
\eqref{eq:spacetime-duality-convention}. On the Lorentzian slice
$t,y^1,x^2,x^3\in\R$, with $y^2=x^2+ix^3$, $y^3=x^2-ix^3$ and
$\varrho=|\vec x|>0$, its diagonal curvature is the complex chiral
Coulomb field
\[
 \mathsf f=\frac{dt\wedge d\varrho}{2\varrho^2}
   -\frac{i}{2\varrho^3}\vec x\mathbin{\cdot}d\vec S,\qquad
 \vec x\mathbin{\cdot}d\vec S
 =y^1dx^2\wedge dx^3+x^2dx^3\wedge dy^1+x^3dy^1\wedge dx^2.
\]
The root coordinates
$r=(x^2-ix^3)/(\varrho+y^1)$ and
$s=(x^2+ix^3)/(\varrho-y^1)$ cover the complement of the worldline.
Corollary~\ref{cor:euler-root-charts} supplies the gauge transition;
the respective polar-axis singularities are chart singularities.
For the real volume form $dt\wedge dy^1\wedge dx^2\wedge dx^3$,
the Lorentzian Hodge equation is $*_LF_A=-iF_A$.

The full connection admits no positive-definite parallel Hermitian
metric on an open subset of a maximally totally real four-dimensional
slice. Such a metric would supply a parallel orthogonal complement to
the known parallel line; the curvature obstruction persists because the
complexified tangent space of that slice is the ambient complex tangent
space. Thus this example does not define a compact-unitary Coulomb
background. No assertion about indefinite real forms is needed here.
\end{Remark}

\begin{Corollary}[Identical profiles and distinct Ward connections]
\label{cor:euler-coulomb-same-profile}
Let $\mathcal V$ be the degree-one Euler modification and let
$\mathcal V_{\rm split}=\mathcal C_{-1}\oplus\mathcal C_1$ be the
diagonal Coulomb modification.  Their filtered profiles are equal:
\[
 \Xi_B(\mathcal V;t)=\Xi_B(\mathcal V_{\rm split};t)
   =[\calO_B(-1)]t^{-1}+[\calO_B(1)]t.
\]
Their stable residues vanish and all the profile-based cyclic moments
of Section~\ref{sec:moments} agree.  This equality persists after
applying any algebraic representation of $SL_2$.
Nevertheless their rank-two Ward connections are locally inequivalent: the Euler
connection is reducible and indecomposable, while the diagonal Coulomb
connection splits into parallel lines.
\end{Corollary}
\begin{proof}
Theorem~\ref{thm:SL2-family} gives the Euler profile at $d=1$.
For the split modification, \eqref{eq:principal-symmetric} and
\eqref{eq:normal-filtration} give the same two graded line bundles.
The residue and moments are functions of this profile. By
Corollary~\ref{cor:cyclic-source-exhaustion}, every periodic cyclic
pairing on the supported classes of its coefficients also agrees.  For arbitrary
representations, the profile maps are ring homomorphisms by
Proposition~\ref{prop:tensor-profile}, and $R(SL_2)$ is generated by the
standard representation: its symmetric powers satisfy
$[\operatorname{Sym}^{n+1}V]=[V][\operatorname{Sym}^{n}V]
 -[\operatorname{Sym}^{n-1}V]$ for $n\geq1$.
The nonzero curvature commutator
\eqref{eq:euler-curvature-commutator} proves the last assertion, since
the curvature endomorphisms of the diagonal connection commute.
\end{proof}

\begin{Remark}[Curvature endomorphism algebras]
\label{rem:curvature-algebras}
Let $E$ be the rank-two Ward bundle and let
$\mathscr R_x(A)\subset\End(E_x)$ be the unital algebra generated by all
curvature endomorphisms at $x$. On the finite-root domain these lie in
$\operatorname{span}\{\mathsf H,\mathsf N\}$ by \eqref{eq:euler-curvature},
and the proof of Proposition~\ref{prop:euler-field} gives the values
$\mathsf H/(2\varrho^2)$ and $\mathsf N/\varrho$. On the chart of
Corollary~\ref{cor:euler-root-charts}, including $s=0$,
$F_{A_\infty}=\mathsf H\,d\mathsf a_\infty+2\mathsf N\,\mathsf a_\infty\wedge ds$
has $F_{A_\infty}(\partial_t,\partial_\varrho)=\mathsf H/(2\varrho^2)$ and
$F_{A_\infty}(\partial_s,\partial_t)=-\mathsf N/\varrho$. Hence, on the
chosen-root cover with $\varrho\ne0$,
\[
 \mathscr R_x(A_{\rm Euler})
   \cong\left\{\begin{pmatrix}a&0\\b&c\end{pmatrix}:a,b,c\in\C\right\},
 \qquad
 \mathscr R_x(A_{\rm split})\cong\C\oplus\C.
\]
These algebras, which are conjugated by gauge transformations, have
centres of dimensions one and two respectively; the distinction is an
algebraic restatement of the curvature obstruction and uses the
reconstructed connection.
\end{Remark}

\section{Foliation realization and transverse source classes}
\label{sec:tangent}
\label{sec:foliation}
We use an elementary holonomy realization of the doubled-divisor
groupoid to fix the transverse conventions and compare the inherited
and relative orientations. This realizes the familiar distinction
between a Hausdorff holonomy groupoid and a non-Hausdorff leaf space.

\begin{Proposition}[Holonomy realization of the double]
\label{prop:foliation-realization}
Let $X$ be a second-countable Hausdorff smooth manifold without boundary
and $D\subset X$ a closed smooth submanifold. Put
\[
 \mathscr M=(X\times\R)\setminus(D\times\{0\}),\qquad
 \pi(x,t)=x,\qquad \mathcal F=\ker(d\pi).
\]
Then $\mathcal F$ is a regular one-dimensional foliation and
$T=X\times\{-1,+1\}$ is a complete transversal, with
\begin{equation}\label{eq:holonomy-double}
 \operatorname{Hol}(\mathscr M,\mathcal F)|_T\cong\calG_{X,D},\qquad
 C_r^*(\mathscr M,\mathcal F)\sim_{\mathrm{Morita}}A_{X,D}.
\end{equation}
\end{Proposition}
\begin{proof}
The manifold $\mathscr M$ is open in $X\times\R$ and $\pi$ is a submersion.
Above $x\in U=X\setminus D$, the leaf is $\{x\}\times\R$;
above $x\in D$, the two leaves are the positive and negative
half-lines. These leaves are simply connected, and projection to $X$
gives transverse coordinates with identity holonomy. The full holonomy
groupoid has the open Hausdorff arrow manifold
\[
 H=\{(x,t,u)\in X\times\R^2:x\in U\ \text{or}\ tu>0\},
\]
endowed with the subspace topology from $X\times\R^2$, with source
$(x,u)$, range $(x,t)$ and composition
$(x,t,u)(x,u,v)=(x,t,v)$. Indeed, the vertical straight segment joins
any two points of one leaf and represents its unique path class;
foliation charts identify this class smoothly by its endpoints.
Every leaf meets $T$. Restriction leaves two identity components $X_\pm$
and two cross-sheet components $U_{+-},U_{-+}$, with precisely the
structure maps and topology of $\calG_{X,D}$.
The complete transversal gives a groupoid equivalence
\cite{CrainicMoerdijk2001}. Lebesgue densities along the leaves supply
a Haar system on $H$, and the transversal groupoid has counting measure.
The reduced equivalence theorem
\cite[published Theorem~4.1, p.~234]{SimsWilliams2012}
proves the second assertion.
\end{proof}

Here $\nu\mathcal F=T\mathscr M/\mathcal F\cong \pi^*T_\R X$.
For $X=X_0$, the manifold $\mathscr M$ has real dimension seven and the foliation
has codimension six. The same construction applies to $\overline X$,
but its foliated manifold $\mathscr M$ is still noncompact.

\begin{Proposition}[Ambient and divisor transverse classes]
\label{prop:transverse-source-classes}
In the physical model, use the transverse cocycle normalization of
Section~\ref{sec:induction}.
\begin{enumerate}[label=(\roman*)]
\item Give $\mathcal F$ the transverse complex orientation of $X_0$.
With this normalization, its transverse fundamental cocycle on $T$ is $\Upsilon_3$ of
\eqref{eq:ambient-six-cocycle}, and
\begin{equation}\label{eq:ambient-source-vanishing}
 \langle\Upsilon_3,\Src(\xi)\rangle=0
 \qquad(\xi\in K^0(B)).
\end{equation}
\item The saturated submanifold
$\mathscr M_D=\pi^{-1}(D_0)=D_0\times\R^*$ carries the restricted foliation.
The complex normal orientation induces the same complex transverse
orientation of $D_0$ on both components. Its fundamental cocycle,
pulled back by $(\rho_+,\rho_-)$, is
$\Phi^{\mathrm{sum}}=\Phi^++\Phi^-$. In the coordinates of
\eqref{eq:Thom-coordinates},
\begin{equation}\label{eq:inherited-defect-source-vanishing}
 \langle\Phi^{\mathrm{sum}},z\rangle=2x_0,
 \qquad
 \langle\Phi^{\mathrm{sum}},\Src(\xi)\rangle=0.
\end{equation}
\item If the negative component is instead given the opposite
transverse orientation, the pulled-back fundamental cocycle is
$\Phi^{\rel}$ and
\begin{equation}\label{eq:defect-source-degree}
 \frac12\langle\Phi^{\rel},\Src(\calO_B(d))\rangle=d.
\end{equation}
This opposite orientation is additional relative data, not the
orientation inherited from the ambient transverse complex orientation.
\end{enumerate}
\end{Proposition}
\begin{proof}
The two sheets of $T$ carry the same orientation, preserved by the
cross-sheet holonomy over $U$. The transverse differential-form
convolution cycle \cite[Section~4]{ConnesTF1986} consists of matrices of forms
with entrywise differential and the support conditions of
\eqref{eq:smooth-core}. Integration on its units is
$\int_{X_0}\Tr$, giving exactly \eqref{eq:ambient-six-cocycle}.
All products have compact support; the off-diagonal entries vanish
near $D_0$, as do their derivatives. The defining property
$\iota_*\Src(\xi)=0$ of supported induction immediately gives
\eqref{eq:ambient-source-vanishing}.

The complete-transversal reduction for $\mathscr M_D$ is the disjoint union
of the unit groupoids on $D_{0,+}$ and $D_{0,-}$. The inherited
orientations give the sum of the two integration cocycles.
Equation~\eqref{eq:individual-pairing-values} gives $2x_0$, and
supported classes have $x=0$. Reversing the negative component's
orientation changes the sum to
$\rho_+^*\tau_{1,2}-\rho_-^*\tau_{1,2}=\Phi^{\rel}$.
Since $\rho_{\pm*}\Src(\xi)=\pm\beta(\xi)$,
Corollary~\ref{cor:rank-degree} gives \eqref{eq:defect-source-degree}.
There is no holonomy between the two components of $\mathscr M_D$,
so their orientations can be assigned independently; the restricted
foliation does not select this relative sign.
\end{proof}

Thus the inherited ambient and divisor transverse fundamental classes
annihilate supported source classes. The relative degree pairing uses
an additional orientation choice and the specified source class.
The full cyclic calculation makes this distinction precise for all
sheet-invariant periodic classes, as follows.

\begin{Corollary}[Sheet symmetry and source detection]
\label{cor:cyclic-sheet-symmetry}
Let $\sigma$ exchange the sheets and write
$\mathcal P=HP^0_{\mathrm{cont}}(\alg_{X_0,D_0})$ for the theory of
Theorem~\ref{thm:full-periodic-cyclic}. Its eigenspaces are
\[
 \begin{aligned}
 \mathcal P_+&=\ker(\sigma^*-1)
   =\operatorname{span}_{\C}\{[\Upsilon_2],[\Upsilon_3]\},\\
 \mathcal P_-&=\ker(\sigma^*+1)
   =\operatorname{span}_{\C}
       \{\tfrac12[\Psi^{\rel}],\tfrac12[\Phi^{\rel}]\}.
 \end{aligned}
\]
The annihilator of the supported source subgroup in $\mathcal P$ is
exactly $\mathcal P_+$. Restriction of the pairing gives an isomorphism
\[
 \mathcal P_-\xrightarrow{\ \cong\ }
 \Hom_{\C}(K^0(B)\otimes\C,\C),\qquad
 [\tau]\longmapsto
 \bigl(\xi\longmapsto\langle[\tau],\Src(\xi)\rangle\bigr).
\]
In particular every sheet-invariant periodic cyclic class annihilates
supported sources. The inherited divisor sums satisfy
\begin{equation}\label{eq:inherited-divisor-periodic-classes}
 [\Phi^+]+[\Phi^-]=2[\Upsilon_2],\qquad
 [\Psi^+]+[\Psi^-]=0
 \quad\text{in }\mathcal P.
\end{equation}
\end{Corollary}
\begin{proof}
Sheet exchange is conjugation by a constant permutation matrix after
ambient inclusion. The full matrix trace makes $\sigma^*\Upsilon_j=
\Upsilon_j$ as cochains, while $\rho_\pm\circ\sigma=\rho_\mp$ gives
$\sigma^*\Psi^{\rel}=-\Psi^{\rel}$ and
$\sigma^*\Phi^{\rel}=-\Phi^{\rel}$. The basis in
Theorem~\ref{thm:full-periodic-cyclic} proves the eigenspace formulas.
Corollary~\ref{cor:source-eigenspaces} gives
$\sigma_*\Src(\xi)=-\Src(\xi)$, so naturality of the pairing makes
every invariant class vanish on supported sources. Conversely,
Corollary~\ref{cor:rank-degree} identifies the restrictions of the two
half-relative classes with the rank and degree basis of the displayed
dual, proving the annihilator and isomorphism assertions. Finally,
\eqref{eq:individual-pairing-values} makes the two sides of each identity
in \eqref{eq:inherited-divisor-periodic-classes} pair equally with every
$K_0(A_0)$-class. Surjectivity of the Chern--Connes character in
Theorem~\ref{thm:full-periodic-cyclic} therefore makes their difference
evaluate to zero on all of $HP_0(\alg_{X_0,D_0})$.
The injective evaluation map
\eqref{eq:periodic-evaluation-injective} makes that difference zero in
$HP^0_{\mathrm{cont}}(\alg_{X_0,D_0})$, proving the class identities.
\end{proof}

The inherited transverse classes lie in the annihilator. Opposite
defect orientations select an anti-invariant detector. Simultaneous
sheet exchange changes the signs of both the relative cocycle and the
supported class, so their pairing is unchanged; this covariance does
not make the relative cocycle the inherited fundamental class.

The NCFI construction
\cite[Section~1.1, Definitions~2--3, pp.~2--5]{ZoisNCFI2026}
assumes a compact foliated manifold and a designated, smoothly
representable $K$-class obtained from its transverse geometric module.
Our realization is noncompact, even for $X=\overline X$, and no such
module class is constructed or identified with $\Src(\xi)$. This
section defines and computes neither the NCFI nor a relative NCFI.
It compares the displayed transverse cyclic classes with the supported
source construction.

\begin{Remark}[Compact tangent comparison]\label{prop:tangent-zero}
Put $E=T_\R\overline X\otimes\C$. The pairing of
$\kappa(E,E)\in K_0(\overline A)$ with one half of the ambient degree-six
matrix-trace cocycle is zero. Indeed,
$E=T^{1,0}\overline X\oplus T^{0,1}\overline X$, and
$\ch_j(\overline E)=(-1)^j\ch_j(E)$ makes its degree-six character vanish
\cite{BottTu1982,MilnorStasheff1974}. This is the standard cancellation
for a complexified real bundle.
\end{Remark}

\section{Noncommutative source information and outlook}
\label{sec:outlook}
The analytic and geometric computations give a defined answer to the
original NCG question. The doubled algebra retains its two divisor
evaluations, the integral attachment obstruction of
Section~\ref{sec:K-theory}, and a four-dimensional periodic cyclic theory.
Theorems~\ref{thm:complete-pairing} and~\ref{thm:full-periodic-cyclic}
identify its integral $K$-lattice and exhaust its continuous periodic
cohomology classes. Its supported source subgroup has exactly two independent
additive coordinates, rank and degree. Choosing a different periodic
cocycle on this algebra supplies no additional stable source observable.

The holomorphic tests locate the information lost in forming these
classes. For pullback quotients with generically invertible second
fundamental form, its torsion cokernel retains the mixed-line jumping locus, while its stable
class retains only the total length. More strongly,
Corollary~\ref{cor:euler-coulomb-same-profile} gives identical full
filtered $K$-profiles for the Euler modification and the diagonal Coulomb
embedding, in every algebraic $SL_2$ representation. Their rank-two Ward connections are
respectively indecomposable and split. The explicit gauge transition
and nonzero curvature commutator establish this distinction on the
chosen-root cover. A source invariant separating these examples must
retain data beyond the present stable profiles, such as their
holomorphic extensions or reconstructed connections.
Remark~\ref{rem:curvature-algebras} illustrates the latter possibility:
the curvature endomorphism algebras have centres of different dimensions,
an observable that requires the reconstructed field as input.

In the holonomy realization, the inherited ambient and divisor
transverse fundamental classes annihilate the supported subgroup.
Corollary~\ref{cor:cyclic-sheet-symmetry} identifies its entire periodic
annihilator as the sheet-invariant sector. The relative degree class
$\Phi^{\rel}$ uses an additional choice of opposite defect orientations.
An NCFI construction would further require an admissible foliated setting
and its designated transverse geometric module class; neither an NCFI
nor a relative NCFI is computed here. These results compare analytic
source classes with relative-cohomological and holomorphic sourced-field
constructions \cite{OConnell2024,Bailey1985,BaileySinger1989,
GarnerPaquette2025}. A stronger comparison must retain the extension or
connection data together with the prescribed support and orientations.


\begin{thebibliography}{99}

\bibitem{AlvarezPaivaDuran2009}
\'Alvarez Paiva, J.~C., and Dur\'an, C.~E.,
Geometric invariants of fanning curves,
\emph{Adv. Appl. Math.} \textbf{42} (2009), 290--312.
\href{https://doi.org/10.1016/j.aam.2006.07.008}{doi:10.1016/j.aam.2006.07.008};
\href{https://arxiv.org/abs/math/0502481}{arXiv:math/0502481}.

\bibitem{AnHuefKumjianSims2011}
an~Huef, A., Kumjian, A., and Sims, A.,
A Dixmier--Douady theorem for Fell algebras,
\emph{J. Funct. Anal.} \textbf{260} (2011), 1543--1581.
\href{https://doi.org/10.1016/j.jfa.2010.11.011}{doi:10.1016/j.jfa.2010.11.011}.

\bibitem{AtiyahKTheory1967}
Atiyah, M.~F.,
\emph{K-Theory},
W.~A. Benjamin, New York, 1967.

\bibitem{AtiyahDrinfeldHitchinManin1978}
Atiyah, M.~F., Drinfeld, V.~G., Hitchin, N.~J., and Manin, Yu.~I.,
Construction of instantons,
\emph{Phys. Lett. A} \textbf{65} (1978), no.~3, 185--187.

\bibitem{AtiyahDunajskiMason2017}
Atiyah, M., Dunajski, M., and Mason, L.~J.,
Twistor theory at fifty: from contour integrals to twistor strings,
\emph{Proc. Roy. Soc. A} \textbf{473} (2017), 20170530.

\bibitem{AtiyahHitchinSinger1978}
Atiyah, M.~F., Hitchin, N.~J., and Singer, I.~M.,
Self-duality in four-dimensional Riemannian geometry,
\emph{Proc. Roy. Soc. Lond. A} \textbf{362} (1978), 425--461.

\bibitem{AtiyahWard1977}
Atiyah, M.~F., and Ward, R.~S.,
Instantons and algebraic geometry,
\emph{Commun. Math. Phys.} \textbf{55} (1977), no.~2, 117--124.

\bibitem{Bailey1985}
Bailey, T.~N.,
Twistors and fields with sources on world-lines,
\emph{Proc. Roy. Soc. Lond. A} \textbf{397} (1985), 143--155.

\bibitem{BaileySinger1989}
Bailey, T.~N., and Singer, M.~A.,
On the twistor description of sourced fields,
\emph{Proc. Roy. Soc. Lond. A} \textbf{422} (1989), 367--385.

\bibitem{BastonEastwood1989}
Baston, R.~J., and Eastwood, M.~G.,
\emph{The Penrose Transform: Its Interaction with Representation Theory},
Oxford Mathematical Monographs, Oxford University Press, 1989.

\bibitem{BeauvilleLaszlo1995}
Beauville, A., and Laszlo, Y.,
Un lemme de descente,
\emph{C. R. Acad. Sci. Paris S\'er. I Math.} \textbf{320} (1995), no.~3, 335--340.

\bibitem{Blackadar1998}
Blackadar, B.,
\emph{K-Theory for Operator Algebras}, second edition,
Mathematical Sciences Research Institute Publications, Vol.~5,
Cambridge University Press, 1998.

\bibitem{BottTu1982}
Bott, R., and Tu, L.~W.,
\emph{Differential Forms in Algebraic Topology},
Graduate Texts in Mathematics, Vol.~82, Springer, 1982.

\bibitem{BrainMajid2008}
Brain, S., and Majid, S.,
Quantisation of twistor theory by cocycle twist,
\emph{Commun. Math. Phys.} \textbf{284} (2008), 713--774.
\href{https://arxiv.org/abs/math/0701893}{arXiv:math/0701893}.

\bibitem{Connes1980}
Connes, A.,
\(C^*\)-alg\`ebres et g\'eom\'etrie diff\'erentielle,
\emph{C. R. Acad. Sci. Paris Ser. A-B} \textbf{290} (1980), A599--A604.

\bibitem{ConnesSurvey1982}
Connes, A.,
A survey of foliations and operator algebras,
\emph{Proc. Sympos. Pure Math.} \textbf{38} (1982), 521--628.

\bibitem{Connes1985}
Connes, A.,
Noncommutative differential geometry,
\emph{Inst. Hautes Etudes Sci. Publ. Math.} \textbf{62} (1985), 41--144.

\bibitem{ConnesTF1986}
Connes, A.,
Cyclic cohomology and the transverse fundamental class of a foliation,
in \emph{Geometric Methods in Operator Algebras} (Kyoto, 1983),
Pitman Research Notes in Mathematics, Vol.~123,
Longman, Harlow, 1986, pp.~52--144.

\bibitem{Connes1994}
Connes, A.,
\emph{Noncommutative Geometry},
Academic Press, San Diego, 1994.

\bibitem{ConnesMoscovici1995}
Connes, A., and Moscovici, H.,
The local index formula in noncommutative geometry,
\emph{Geom. Funct. Anal.} \textbf{5} (1995), 174--243.

\bibitem{ConnesSkandalis1984}
Connes, A., and Skandalis, G.,
The longitudinal index theorem for foliations,
\emph{Publ. Res. Inst. Math. Sci.} \textbf{20} (1984), 1139--1183.

\bibitem{Crainic1999}
Crainic, M.,
Cyclic cohomology of \'etale groupoids: the general case,
\emph{K-Theory} \textbf{17} (1999), 319--362.
\href{https://arxiv.org/abs/funct-an/9712001v1}{arXiv:funct-an/9712001v1}.

\bibitem{CrainicMoerdijk2000}
Crainic, M., and Moerdijk, I.,
A homology theory for \'etale groupoids,
\emph{J. Reine Angew. Math.} \textbf{521} (2000), 25--46.
\href{https://arxiv.org/abs/math/9905011v1}{arXiv:math/9905011v1}.

\bibitem{CrainicMoerdijk2001}
Crainic, M., and Moerdijk, I.,
Foliation groupoids and their cyclic homology,
\emph{Adv. Math.} \textbf{157} (2001), 177--197.
\href{https://arxiv.org/abs/math/0003119v2}{arXiv:math/0003119v2}.

\bibitem{Cuntz1997}
Cuntz, J.,
Bivariant \(K\)-theory for locally convex algebras and the Chern--Connes character,
\emph{Doc. Math.} \textbf{2} (1997), 139--182.

\bibitem{CuntzQuillen1995}
Cuntz, J., and Quillen, D.,
Cyclic homology and nonsingularity,
\emph{J. Amer. Math. Soc.} \textbf{8} (1995), 373--442.

\bibitem{CuntzQuillen1997}
Cuntz, J., and Quillen, D.,
Excision in bivariant periodic cyclic cohomology,
\emph{Invent. Math.} \textbf{127} (1997), 67--98.

\bibitem{Davidson1996}
Davidson, K.~R.,
\emph{\(C^*\)-Algebras by Example},
Fields Institute Monographs, Vol.~6, American Mathematical Society, 1996.

\bibitem{DeeleyGoffengYashinski2022}
Deeley, R.~J., Goffeng, M., and Yashinski, A.,
Fell algebras, groupoids, and projections,
\emph{Proc. Amer. Math. Soc.} \textbf{150} (2022), 4891--4907.
\href{https://doi.org/10.1090/proc/16045}{doi:10.1090/proc/16045};
\href{https://arxiv.org/abs/2109.10820}{arXiv:2109.10820}.

\bibitem{EastwoodPenroseWells1981}
Eastwood, M.~G., Penrose, R., and Wells, R.~O., Jr.,
Cohomology and massless fields,
\emph{Comm. Math. Phys.} \textbf{78} (1981), 305--351.

\bibitem{EmersonMeyer2009}
Emerson, H., and Meyer, R.,
Equivariant representable $K$-theory,
\emph{J. Topol.} \textbf{2} (2009), 123--156.
\href{https://doi.org/10.1112/jtopol/jtp003}{doi:10.1112/jtopol/jtp003};
\href{https://arxiv.org/abs/0710.1410}{arXiv:0710.1410}.

\bibitem{Fulton1998}
Fulton, W.,
\emph{Intersection Theory}, second edition,
Ergebnisse der Mathematik und ihrer Grenzgebiete, Vol.~2,
Springer, 1998.

\bibitem{FultonHarris1991}
Fulton, W., and Harris, J.,
\emph{Representation Theory: A First Course},
Graduate Texts in Mathematics, Vol.~129, Springer-Verlag, New York, 1991.

\bibitem{GarnerPaquette2025}
Garner, N., and Paquette, N.~M.,
Scattering off of twistorial line defects,
\emph{J. High Energy Phys.} \textbf{2025}, no.~5, article no.~228.
\href{https://doi.org/10.1007/JHEP05(2025)228}{doi:10.1007/JHEP05(2025)228};
\href{https://arxiv.org/abs/2408.11092}{arXiv:2408.11092}.

\bibitem{GoddardNuytsOlive1977}
Goddard, P., Nuyts, J., and Olive, D.~I.,
Gauge theories and magnetic charge,
\emph{Nucl. Phys. B} \textbf{125} (1977), 1--28.

\bibitem{Grothendieck1957}
Grothendieck, A.,
Sur la classification des fibr\'es holomorphes sur la sph\`ere de Riemann,
\emph{Amer. J. Math.} \textbf{79} (1957), 121--138.
\href{https://doi.org/10.2307/2372388}{doi:10.2307/2372388}.

\bibitem{Hartshorne1977}
Hartshorne, R.,
\emph{Algebraic Geometry},
Graduate Texts in Mathematics, Vol.~52, Springer, 1977.

\bibitem{Haslehurst2021}
Haslehurst, M.,
Relative $K$-theory for $C^*$-algebras,
\href{https://arxiv.org/abs/2106.02620}{arXiv:2106.02620} (2021).

\bibitem{HughstonMason1990}
Hughston, L.~P., and Mason, L.~J. (eds.),
\emph{Further Advances in Twistor Theory. Volume I: The Penrose Transform and its Applications},
Pitman Research Notes in Mathematics Series, Vol.~231,
Longman, Harlow, 1990.

\bibitem{KapustinWitten2007}
Kapustin, A., and Witten, E.,
Electric-magnetic duality and the geometric Langlands program,
\emph{Commun. Number Theory Phys.} \textbf{1} (2007), 1--236.

\bibitem{Khalkhali2009}
Khalkhali, M.,
\emph{Basic Noncommutative Geometry},
EMS Series of Lectures in Mathematics, European Mathematical Society, 2009.

\bibitem{Loday1998}
Loday, J.-L.,
\emph{Cyclic Homology}, second edition,
Grundlehren der mathematischen Wissenschaften, Vol.~301,
Springer, 1998.

\bibitem{MarcolliPenrose2020}
Marcolli, M., and Penrose, R.,
Gluing Non-commutative Twistor Spaces,
\emph{Q. J. Math.} \textbf{72} (2021), nos.~1--2, 417--454.
\href{https://doi.org/10.1093/qmath/haab024}{doi:10.1093/qmath/haab024};
\href{https://arxiv.org/abs/2012.02823v1}{arXiv:2012.02823v1}.

\bibitem{MasonWoodhouse1996}
Mason, L.~J., and Woodhouse, N.~M.~J.,
\emph{Integrability, Self-Duality, and Twistor Theory},
London Mathematical Society Monographs, New Series, Vol.~15,
Oxford University Press, 1996.

\bibitem{MeierViana2026}
Meier, T., and Viana, E.,
Twisting twistor space,
\emph{J. High Energy Phys.} \textbf{2026}, no.~10, article no.~130.
\href{https://doi.org/10.1007/JHEP10(2026)130}{doi:10.1007/JHEP10(2026)130};
\href{https://arxiv.org/abs/2601.19983}{arXiv:2601.19983}.

\bibitem{MilnorStasheff1974}
Milnor, J.~W., and Stasheff, J.~D.,
\emph{Characteristic Classes},
Annals of Mathematics Studies, Vol.~76,
Princeton University Press, 1974.

\bibitem{MirkovicVilonen2007}
Mirkovi\'c, I., and Vilonen, K.,
Geometric Langlands duality and representations of algebraic groups over commutative rings,
\emph{Ann. of Math. (2)} \textbf{166} (2007), no.~1, 95--143.

\bibitem{Moerdijk1995}
Moerdijk, I.,
\emph{Classifying Spaces and Classifying Topoi},
Lecture Notes in Mathematics, Vol.~1616, Springer, 1995.

\bibitem{MoerdijkMrcun2003}
Moerdijk, I., and Mr\v{c}un, J.,
\emph{Introduction to Foliations and Lie Groupoids},
Cambridge Studies in Advanced Mathematics, Vol.~91,
Cambridge University Press, 2003.

\bibitem{MuhlyRenaultWilliams1987}
Muhly, P.~S., Renault, J.~N., and Williams, D.~P.,
Equivalence and isomorphism for groupoid \(C^*\)-algebras,
\emph{J. Operator Theory} \textbf{17} (1987), 3--22.

\bibitem{OConnell2024}
O'Connell, D.,
Vector bundles over non-Hausdorff manifolds,
\emph{Topology Appl.} \textbf{354} (2024), 108982.
\href{https://doi.org/10.1016/j.topol.2024.108982}{doi:10.1016/j.topol.2024.108982};
\href{https://arxiv.org/abs/2306.14117}{arXiv:2306.14117}.

\bibitem{Paterson1999}
Paterson, A.~L.~T.,
\emph{Groupoids, Inverse Semigroups, and their Operator Algebras},
Progress in Mathematics, Vol.~170, Birkh\"auser, 1999.

\bibitem{Penrose1967}
Penrose, R.,
Twistor algebra,
\emph{J. Math. Phys.} \textbf{8} (1967), 345--366.

\bibitem{Penrose1976}
Penrose, R.,
Nonlinear gravitons and curved twistor theory,
\emph{Gen. Relativity Gravitation} \textbf{7} (1976), 31--52.

\bibitem{PenroseRindler1}
Penrose, R., and Rindler, W.,
\emph{Spinors and Space-Time. Vol. 1: Two-Spinor Calculus and Relativistic Fields},
Cambridge Monographs on Mathematical Physics, Cambridge University Press, 1984.

\bibitem{PenroseRindler2}
Penrose, R., and Rindler, W.,
\emph{Spinors and Space-Time. Vol. 2: Spinor and Twistor Methods in Space-Time Geometry},
Cambridge Monographs on Mathematical Physics, Cambridge University Press, 1986.

\bibitem{PenroseSparling1979}
Penrose, R., and Sparling, G.~A.~J.,
The anti-self-dual Coulomb field's non-Hausdorff twistor space,
\emph{Twistor Newsletter} \textbf{9} (1979); reprinted in \cite{HughstonMason1990}.

\bibitem{PressleySegal1986}
Pressley, A., and Segal, G.,
\emph{Loop Groups},
Oxford Mathematical Monographs, Oxford University Press, 1986.

\bibitem{RaeburnWilliams1998}
Raeburn, I., and Williams, D.~P.,
\emph{Morita Equivalence and Continuous-Trace \(C^*\)-Algebras},
Mathematical Surveys and Monographs, Vol.~60,
American Mathematical Society, 1998.

\bibitem{Renault1980}
Renault, J.,
\emph{A Groupoid Approach to \(C^*\)-Algebras},
Lecture Notes in Mathematics, Vol.~793, Springer, 1980.

\bibitem{Schochet1984}
Schochet, C.,
Topological methods for $C^*$-algebras. III. Axiomatic homology,
\emph{Pacific J. Math.} \textbf{114} (1984), 399--445.
\href{https://doi.org/10.2140/pjm.1984.114.399}{doi:10.2140/pjm.1984.114.399}.

\bibitem{SimsWilliams2012}
Sims, A., and Williams, D.~P.,
Renault's equivalence theorem for reduced groupoid $C^*$-algebras,
\emph{J. Operator Theory} \textbf{68} (2012), 223--239.

\bibitem{StacksRiemannHurwitz}
The Stacks Project Authors,
Riemann--Hurwitz, Section~53.12, Tag~0C1B,
\emph{The Stacks Project}.
\url{https://stacks.math.columbia.edu/tag/0C1B}.

\bibitem{Tu2004}
Tu, J.-L.,
Non-Hausdorff groupoids, proper actions and \(K\)-theory,
\emph{Doc. Math.} \textbf{9} (2004), 565--597.

\bibitem{Ward1977}
Ward, R.~S.,
On self-dual gauge fields,
\emph{Phys. Lett. A} \textbf{61} (1977), 81--82.

\bibitem{Ward1985}
Ward, R.~S.,
Integrable and solvable systems, and relations among them,
\emph{Philos. Trans. R. Soc. Lond. A} \textbf{315} (1985), 451--457.

\bibitem{WardWells1990}
Ward, R.~S., and Wells, R.~O., Jr.,
\emph{Twistor Geometry and Field Theory},
Cambridge Monographs on Mathematical Physics, Cambridge University Press, 1990.

\bibitem{Williams2007}
Williams, D.~P.,
\emph{Crossed Products of \(C^*\)-Algebras},
Mathematical Surveys and Monographs, Vol.~134,
American Mathematical Society, 2007.

\bibitem{WoodhouseMason1988}
Woodhouse, N.~M.~J., and Mason, L.~J.,
The Geroch group and non-Hausdorff twistor spaces,
\emph{Nonlinearity} \textbf{1} (1988), 73--114.

\bibitem{Zhu2016}
Zhu, X.,
An introduction to affine Grassmannians and the geometric Satake equivalence,
in \emph{Geometry of Moduli Spaces and Representation Theory},
IAS/Park City Mathematics Series, Vol.~24,
American Mathematical Society, Providence, RI, 2017, pp.~59--154.
\href{https://doi.org/10.1090/pcms/024/02}{doi:10.1090/pcms/024/02};
\href{https://arxiv.org/abs/1603.05593v2}{arXiv:1603.05593v2}.

\bibitem{Ziegler2015}
Ziegler, P.,
Graded and filtered fiber functors on Tannakian categories,
\emph{J. Inst. Math. Jussieu} \textbf{14} (2015), 87--130.
\href{https://doi.org/10.1017/S1474748013000376}{doi:10.1017/S1474748013000376};
\href{https://arxiv.org/abs/1111.1981}{arXiv:1111.1981}.

\bibitem{Zois2000}
Zois, I.~P.,
A New Invariant for $\sigma$ Models,
\emph{Commun. Math. Phys.} \textbf{209} (2000), 757--783.
\href{https://doi.org/10.1007/PL00020964}{doi:10.1007/PL00020964};
\href{https://arxiv.org/abs/hep-th/9904001}{arXiv:hep-th/9904001}.

\bibitem{ZoisNCFI2026}
Zois, I.~P.,
The Noncommutative Foliation Invariant (NCFI): extension to the odd codimension
case and computed examples,
\href{https://arxiv.org/abs/2607.05430v2}{arXiv:2607.05430v2} (2026).

\end{thebibliography}
\end{document}